\documentclass[leqno]{amsart}

\usepackage{cite}
\usepackage{graphicx,subfigure,xcolor}

\numberwithin{equation}{section}

\usepackage[latin1]{inputenc}
\usepackage[english]{babel}

\usepackage{amsmath,amsthm,amsfonts,latexsym,amssymb,mathtools}
\usepackage[colorlinks]{hyperref}
\hypersetup{linkcolor=blue,citecolor=blue,filecolor=black,urlcolor=blue}
\usepackage{comment}

\usepackage{color}

\usepackage[normalem]{ulem}

{ \theoremstyle{plain}
\newtheorem{theorem}{Theorem}[section]
\newtheorem{proposition}[theorem]{Proposition}
\newtheorem{lemma}[theorem]{Lemma}
\newtheorem{corollary}[theorem]{Corollary}
  \theoremstyle{remark}
\newtheorem{remark}[theorem]{Remark}
  \theoremstyle{definition}
\newtheorem{definition}[theorem]{Definition}
}

\newenvironment{properties}[1]{\begin{enumerate}
\renewcommand{\theenumi}{$(#1_{\arabic{enumi}})$}
\renewcommand{\labelenumi}{$(#1_{\arabic{enumi}})$}}{\end{enumerate}}

\def\R{\mathbb{R}}

\def\N{\mathbb{N}}

\begin{document}


\title[]{An Orlicz space approach to exponential elliptic problems in higher dimensions}

\author[A. Boscaggin]{Alberto Boscaggin}
\address{Alberto Boscaggin\newline\indent
Dipartimento di Matematica 
\newline\indent
Universit\`a di Torino
\newline\indent
via Carlo Alberto 10, 10123 Torino, Italy}
\email{alberto.boscaggin@unito.it}

\author[F. Colasuonno]{Francesca Colasuonno}
\address{Francesca Colasuonno\newline\indent
Dipartimento di Matematica
\newline\indent
Universit\`a di Torino
\newline\indent
via Carlo Alberto 10, 10123 Torino, Italy}
\email{francesca.colasuonno@unito.it}

\author[B. Noris]{Benedetta Noris}
\address{Benedetta Noris \newline \indent 
Dipartimento di Matematica \newline\indent
Politecnico di Milano \newline\indent
p.zza Leonardo da Vinci 32, 20133 Milano, Italy}
\email{benedetta.noris@polimi.it}

\author[F. Sani]{Federica Sani}
\address{Federica Sani\newline \indent
Dipartimento di Scienze Fisiche, Informatiche e Matematiche\newline \indent
Universit\`a di Modena e Reggio Emilia\newline \indent
Via Campi 213/A, 41125 Modena, Italy}
\email{federica.sani@unimore.it}

\begin{abstract} 
We consider semilinear elliptic problems of the form
\[
-\Delta u + \lambda u = f(x,u), \quad u\in H^1_0(A),
\]
where $A\subset\mathbb{R}^N$, $N\geq3$, is either a bounded or unbounded annulus, and $\lambda \geq0$.
We study a broad class of nonlinearities $f$ with superlinear growth at infinity,  including exponential- and power-type ones. Under suitable assumptions, we establish the existence of a positive nonradial solution via techniques in the spirit of Szulkin's nonsmooth critical point theory, applied within a convex cone in Orlicz spaces. Notably, the Trudinger-Moser inequality fails in the whole Sobolev space $H^1_0(A)$.
\end{abstract}

\maketitle

\section{Introduction}

In this paper, we are interested in proving the existence of solutions to semilinear elliptic problems of the form
\begin{equation}\label{eq:main}
-\Delta u + \lambda u = f(x,u), \quad u\in H^1_0(\Omega),
\end{equation}
where $\Omega$ is a domain in $\R^N$, $\lambda \geq0$ and $f$ is a function with superlinear behavior at infinity. 
More precisely, we aim to address the challenge of dealing with exponential-type nonlinearities in dimensions higher than two,  where a Trudinger-Moser inequality is not available.

There are several papers in the literature dealing with such nonlinearities. For instance, the case \( f(x,u) = \varepsilon e^u \) with \( \varepsilon > 0 \) and Neumann boundary conditions corresponds to the well-known Keller-Segel equation, which has been analysed using perturbative techniques as \( \varepsilon \to 0^+ \). In both radial and nonradial bounded domains, solutions concentrating on the boundary have been investigated, see for instance \cite{AgudeloPistoia,PistoiaVaira2015,BonheureCasterasNoris1,BonheureCasterasNoris2, BonheureCasterasFoldes2020}. In contrast, the case of Dirichlet boundary conditions, which corresponds to the Gelfand equation, cannot be addressed with the same techniques. When $\Omega$ is a ball, the structure of the set of singular and regular radial solutions has been examined starting from \cite{JL}, we refer for example to the results \cite{Miy,KW,MN}; see also \cite{NS,Pa} for the case of the annulus.
Let us mention that, still under Neumann boundary conditions, the multiplicity of radial solutions for a large class of exponential nonlinearities has been established in \cite[Corollary 1.3]{BCN1} via shooting techniques.  All these results, however, appear quite far in spirit from ours.

The approach adopted in the present paper is indeed of global variational nature and is closely related to that of \cite{BCNW2023,BCNW2024}, see also \cite{CowanMoameni2022JDE,CM_Trans,CowanMoameni2022}, where the focus is on pure  power nonlinearities in radially symmetric domains.
The equation studied in \cite{BCNW2023, BCNW2024} exhibits a lack of compactness, which can arise either from the pure power nonlinearity,   if $f(x, \cdot)$ is supercritical in the Sobolev embedding sense, or from the unboundedness of the domain. Classical effective tools to overcome the lack of compactness typically exploit specific structural features of the problem, such as symmetry or the presence of potentials/weights.  Inspired by this perspective, in \cite{BNW, ST} the authors propose a method for dealing with supercritical radial problems: instead of working in the natural energy space,  they confine the function set to a closed convex cone, see also \cite{ColasuonnoNoris2017, BCN1, BCN2}. This strategy enables to gain compactness and facilitates the construction of radial solutions with specific monotonicity, even for problems that cannot be treated within the full subspace $H^1_{\mathrm{rad}}$ of radial $H^1$-functions. It should be emphasized that, while restricting the function space to a cone with specific symmetry and monotonicity properties aids in restoring the needed compactness, this approach also introduces certain technical challenges:  working within a set that is not a natural constraint requires additional efforts to ensure the invariance properties of the problem with respect to this set.

The technique introduced in \cite{BNW,ST} has been refined in \cite{BCNW2023,BCNW2024} to prove the existence of nonradial solutions of the Dirichlet problem with a pure power nonlinearity, in radial domains that are respectively bounded and unbounded.
In \cite{CM_Trans,CowanMoameni2022JDE,CowanMoameni2022}, Cowan and Moameni demonstrate how to formalize this method within Szulkin's nonsmooth critical point theory \cite{S} in a natural and efficient framework; more precisely, using the nonsmooth mountain pass theorem, they develop a unified approach to find possibly nonradial solutions in suitable invariant convex cones for elliptic equations with supercritical power-type nonlinearities. 
The advantage of this method is that it is adaptable to various types of problems and allows for the discovery of diverse types of solutions: by varying the cone chosen to frame the problem, in \cite{BCNW2023,BCNW2024} new rotationally nonequivalent solutions are found. 

Depending on the chosen cone, the effective dimension of the problem becomes a smaller integer $2\le n < N$, thus increasing the corresponding critical exponent. For $n>2$, the improved Sobolev critical exponent $2^*_n:=2n/(n-2)$ is finite and the case $p< 2^*_n$ was addressed in \cite{BCNW2024};
for $n=2$,  it is showed in \cite{BCNW2023,BCNW2024} that the invariant cone within which the authors work is embedded in $L^q$ for every $q \in (2, \infty)$. As a result, every power-type nonlinearity behaves as subcritical in the cone. 
Actually, since the dimension $n=2$ corresponds to the limiting case of Sobolev embeddings of $H^1$-functions, one can expect that the power-type growth can be improved, allowing to consider exponential nonlinearities.

\smallskip

Moving from this observation, the present paper aims to provide a unified approach for addressing both exponential- and power-type growth,  dealing with problems on bounded as well as unbounded domains. To maintain a general framework,  we include  nonlinearities subject to assumptions \ref{f1}-\ref{f5} below. 
Our main abstract existence result, Theorem \ref{thm:abstract}, is detailed in Section \ref{sec:stat-main-res}.
Classical approaches to two-dimensional problems with exponential growth (see e.g. \cite{dFMR}) are typically addressed in $H^1_0(\Omega)$, which naturally embeds in the ``small'' Orlicz space $\mathcal{M}(\Omega)$ with exponential Young function $e^{u^2}-1$ as defined in \eqref{eq:M-def}. 
As in dimension $N\ge 3$ such an embedding fails, in order to ensure the well-posedeness of the energy functional, we formulate the problem in the space $V = H^1_0(\Omega) \cap \mathcal{M}(\Omega)$. 
Moreover, to gain compactness, we confine the function set to a closed convex cone $K$.
Some drawbacks arise from this choice of the functional space: we can show that a Palais-Smale-type condition holds only with respect to the $H^1$-topology in the cone $K$, and this prevents the direct application of Szulkin's results \cite{S}.
To overcome this issue, we rely on a suitable version of the nonsmooth mountain pass theorem which provides the existence of a Palais-Smale sequence without compactness hypotheses, see Theorem \ref{thm:ArcoyaBereanuTorres}.
Then, in order to get a solution, we prove that it is possible to pass to the limit in the variational inequality involved in the definition of Szulkin's Palais-Smale sequences. 
Clearly, when working in dimension $N=2$, it holds $V=H^1_0(\Omega) \cap \mathcal{M}(\Omega)=H^1_0(\Omega)$, by the Trudinger-Moser inequality.
Therefore,  our abstract Theorem \ref{thm:abstract} applies with $K=\{u \in H^1_0(\Omega): \, u\geq0\}$, thus recovering some known results in the subcritical exponential case, along the lines of \cite[Theorem 1.1]{dFMR} by de Figueiredo,  Miyagaki, and Ruf.

As an application of the abstract Theorem \ref{thm:abstract}, we consider \eqref{eq:main} in the possibly unbounded annulus
\begin{equation}\label{eq:A-def}
\Omega=A :=\{x\in\mathbb{R}^N:\, R_0<|x|<R_1\}, 
\end{equation}
with $0<R_0 <R_1\le \infty$ and $N\geq3$. For the reader's convenience, let us state here explicitly a consequence of Theorem \ref{thm:abstract} in the case of exponential nonlinearities, while we postpone the more general formulation to the next section.  Let us consider the problem
\begin{equation}\label{eq:main-exp}
\begin{cases}
-\Delta u + \lambda u = u^{\beta-1}(e^{u^\beta}-1) \qquad &\mbox{in }{A}\\
u>0 &\mbox{in }{A}\\
u\in H^1_0(A)
\end{cases}
\end{equation}
with $\lambda > 0$ ($\lambda\ge 0$ if $A$ is bounded) and $\beta\in (1,2)$. 
A direct application of Theorem \ref{thm:abstract} provides a nontrivial solution of \eqref{eq:main-exp} which, due to the symmetries of the problem, may be radial.
In order to ensure the nonradiality of the solution, we rely on a symmetry breaking result, see Theorem \ref{thm_nonradial}, which requires an additional assumption involving the following constant 
\begin{equation}\label{eq:def-H}
H=H(N,\lambda,R_0):=\left(\frac{N-2}{2}\right)^2+\lambda R_0^2
\end{equation}
arising from the Hardy inequality.  Then we have the following result.

\begin{theorem}\label{thm:main-exp}
Let $A$ be as in \eqref{eq:A-def} with $0<R_0 <R_1\le \infty$ and $N\geq3$.  Let $\lambda> 0$ ($\lambda \geq0$ if $A$ is bounded) and
\begin{equation}\label{eq:hp-beta}
\beta \in \left( 1+\frac{N}{H},2\right) \mbox{ if }\lambda=0,
\qquad
\beta \in \left[ 1+\frac{N}{H},2\right) \mbox{ if }\lambda>0.
\end{equation}
Then, there exists a nonradial solution of \eqref{eq:main-exp}.
\end{theorem}

We observe that the conditions in \eqref{eq:hp-beta} do not trivialize whenever $N<H$ ($N\leq H$ if $\lambda$ is positive): this is the case for instance when $N$ or $\lambda R_0^2$ are large.  
In order to show that equation \eqref{eq:main-exp} fits into the framework of the abstract Theorem \ref{thm:abstract}, a crucial intermediate step is to identify a suitable convex cone $\mathcal K$, see Definition \ref{eq:K-def}, and to recover the necessary compactness. Notice that the critical case for problem \eqref{eq:main-exp} in $\mathcal K$ corresponds to the maximal value of $\beta$ allowed for the validity of a Pohozaev-Trudinger-Moser embedding.
We prove a uniform estimate in $\mathcal K$ of Pohozaev-Trudinger-Moser-type with a growth behaving like $e^{u^2}$ at infinity, see Proposition \ref{cor:Trud-Moser1}; this implies that for $\beta < 2$ the equation in \eqref{eq:main-exp} behaves as subcritical in $\mathcal K$,  although supercritical in $H^1_0(A)$.

A direct application of the abstract Theorem \ref{thm:abstract} also leads to recover \cite[Theorems 1.1 and 1.2]{BCNW2023} and \cite[Theorem 1.1]{BCNW2024} with $m=N-1$, \cite[Theorem 1.2]{BCNW2024} with $n=2$, in the case of the pure power nonlinearity $f(x,s)=|s|^{p-2}s$,  $p>2$: working in the small Orlicz space $\mathcal{M}(\Omega)$ allows to fit both the exponential and the power-type growth into the same framework. 

The presentation of these existence and symmetry breaking results in a more general form (see Theorems \ref{thm:applied1} and \ref{thm_nonradial} respectively), as well as the outline of the paper, are postponed to the next section.

\section{Statements of the main results}\label{sec:stat-main-res}{}
In order to state our main results in their full generality, let $\Omega$ be a domain in $\mathbb R^N$, $N\ge 2$, and let us detail the assumptions on the nonlinearity $f$ appearing in \eqref{eq:main}. 
Letting $F(x,s):=\int_0^s f(x,t)\, dt$, we assume:
\begin{properties}{f}
  \item\label{f1} $f\in C(\Omega\times\R;\R)$; $f(x,s)s>0$ for all $x\in \Omega$ and $s\neq 0$; 
  for all $M>0$, there exists $C_M>0$ such that $|f(x,s)|\le C_M$ for every $x\in \Omega$ and $|s|\le M$;\smallskip
  \item \label{f2}  for all $\alpha>0$
  $$\lim_{|s|\to\infty} \frac{f(x,s)}{e^{\alpha s^2}} =0 \quad \text{ uniformly for } x\in\Omega;$$
  \item\label{f3} \begin{itemize}
\item[(b)] \emph{if $\Omega$ is bounded}: there exist $\sigma>2$ and $\overline{M}>0$ such that
$$\sigma F(x,s)\le f(x,s)s \quad \text{ for all } x\in\Omega \text{ and } |s|\ge \overline{M};$$
\item[(u)] \emph{if $\Omega$ is unbounded}: there exists $\sigma>2$ such that
$$\sigma F(x,s)\le f(x,s)s \quad \text{  for all } x\in\Omega \text{ and }s\in\mathbb R;$$
moreover, there exist $0 < c_1\in L^\infty(\Omega)$ and $c_2\in L^1(\Omega)$ such that
$$F(x,s)\ge c_1(x)|s|^\sigma-c_2(x) \quad \text{  for all } x\in\Omega \text{ and }s\in \mathbb R;$$
\end{itemize}
  \item\label{f5} 
\begin{itemize}
\item[(b)] \emph{if $\Omega$ is bounded}: denoting by $\lambda_1$ the first eigenvalue of the Laplacian with homogeneous Dirichlet boundary conditions in $\Omega$, it holds
$$\limsup_{s\to0} \frac{2 F(x,s)}{s^2} < \lambda_1+\lambda \quad \text{ uniformly for } x\in \Omega;$$
\item[(u)] \emph{if $\Omega$ is unbounded}: there exists $\mu>1$ such that
$$\limsup_{s\to0} \frac{|f(x,s)|}{|s|^{\mu}} < \infty \quad \text{ uniformly for } x\in \Omega.$$
\end{itemize}  	
\end{properties}

We refer to Remark \ref{rem:f5u-b} for some comments on the different assumptions required in the bounded and in the unbounded case.

In order to provide a variational setting for \eqref{eq:main}, we introduce the Orlicz space 
\begin{equation*}\label{eq:L-def}
\mathcal L(\Omega):=\left\{ u:\Omega \to \R \text{ measurable}\,:\, \int_\Omega \left( e^{\left(\frac{u}{k}\right)^2} -1 \right)\,dx <\infty \text{ for some } k>0 \right\},
\end{equation*}
which is a Banach space endowed with the norm
\[
\|u\|_{\mathcal L(\Omega)} := \inf\left\{ k>0 \, :\, \int_\Omega \left( e^{\left(\frac{u}{k}\right)^2} -1 \right)\,dx \leq 1 \right\}.
\]
As we recall in Lemma \ref{lem:L-embedded}, $\mathcal L(\Omega)$ is continuously embedded in $L^q(\Omega)$ for every $q\in[2,\infty)$. We consider also its closed subspace\footnote{This space has already been introduced and used in the literature, see for example \cite[Section 3.4, Definition 2]{RR}, but its name and notation are not yet consolidated. In this paper we refer to it as {\it small Orlicz space}.}
\begin{equation}\label{eq:M-def}
\mathcal M(\Omega):=\left\{ u:\Omega \to \R \text{ measurable}\,:\, \int_\Omega \left( e^{\left(\frac{u}{k}\right)^2} -1 \right)\, dx <\infty \text{ for all } k>0 \right\},
\end{equation}
endowed with the same norm. Finally, we let the Banach space $V:= H^1_0(\Omega) \cap \mathcal M(\Omega)$ be equipped with the norm
\begin{equation}\label{eq:normV}
\|u\|_{V}:=\|u\|_{H^1_\lambda(\Omega)} + \|u\|_{\mathcal L(\Omega)},
\end{equation}
where
\[
\|u\|_{H^1_\lambda(\Omega)}:=\left( \int_\Omega \left(|\nabla u|^2+\lambda u^2\right)\,dx\right)^{1/2}.
\]
In the following we will simply write $\|\cdot\|_{H^1(\Omega)}$ for the standard $H^1$-norm $\|\cdot\|_{H^1_1(\Omega)}$. 
The Euler-Lagrange functional associated with \eqref{eq:main}, namely
\begin{equation}\label{eq:J-def}
J(u):=\frac12 \|u\|_{H^1_\lambda(\Omega)}^2-\int_\Omega F(x,u) \, dx,
\end{equation}
is well defined and of class $C^1$ in $V$, see Remark \ref{rmk:JC1}. Nevertheless, as already mentioned in the Introduction, $J$ possibly lacks compactness in the whole space $V$.  

Our first main result is an abstract theorem, extending \cite[Theorem 2.1]{CowanMoameni2022}, which ensures the existence of solutions of \eqref{eq:main} via variational techniques, provided that there exists a subset $K$ of $V$ satisfying suitable assumptions of compactness and invariance. The set $K$ must be a convex cone, meaning that
\begin{itemize}
\item $u\in K$ implies $t u \in K$ for every $t\ge0$;
\item $u,\, -u\in K$ implies $u=0$;
\item $u,\, v\in K$ implies $u+v\in K$.
\end{itemize}
Our abstract theorem can be stated as follows.
\begin{theorem}\label{thm:abstract}
Let $\Omega$ be a domain in $\R^N$, $N\ge 2$, $\lambda >0$ ($\lambda\geq 0$ if $\Omega$ is bounded), and $f$ satisfy \ref{f1}-\ref{f5}. Let $J:V\to\R$ be as in \eqref{eq:J-def}.
Finally, let $K \subset H^1_0(\Omega)$ be a closed convex cone and assume that: 
\begin{itemize}
\item[(i)] \emph{Embeddings:} $K\hookrightarrow \mathcal M(\Omega)$, i.e. $K\subset \mathcal M(\Omega)$ and there exists $C>0$ such that $\|u\|_{\mathcal L(\Omega)}\le C \|u\|_{H^1_\lambda(\Omega)}$ for every $u\in K$;\\
moreover, \emph{if $\Omega$ is unbounded,} $K\hookrightarrow \hookrightarrow L^q(\Omega)$ for every $q> 2$, i.e. if $(u_n)_n \subset K$ is bounded, there exists $u\in K$ such that for every $q>2$, $u_n\to u$ in $L^q(\Omega)$ up to a subsequence; 
\item[(ii)] \emph{Pointwise invariance:} for each $u\in K$ there exists $v\in K$ weak solution of
\[
-\Delta v + \lambda v = f(x,u) \quad \mbox{in }{\Omega}.
\] 
\end{itemize}
Then there exists a nontrivial weak solution $u\in K$ of \eqref{eq:main} such that 
\[
J(u)=\inf_{\gamma\in\Upsilon}\sup_{t\in[0,1]}J(\gamma(t))>0,
\]
with $\Upsilon:=\{\gamma\in C([0,1]; K)\,:\gamma(0)=0	\neq\gamma(1),\,J(\gamma(1))\le 0\}$.
\end{theorem}

\begin{remark}\label{rem:compact-emb-Om-bdd}
Note that assumption (i) of the previous theorem does not explicitly require a compact embedding for $K$ when $\Omega$ is bounded. In this case, the compact embedding $K\hookrightarrow\hookrightarrow L^q(\Omega)$ holds automatically for every $q \in [1,\infty)$ without further requirements. Indeed,  we recall in Lemma \ref{lem:L-embedded} below that $\mathcal M(\Omega)\hookrightarrow L^q(\Omega)$ for every $q\in [1,\infty)$.
Being $K\hookrightarrow \mathcal M(\Omega)$ by assumption (i) of Theorem \ref{thm:abstract}, it follows that $K\hookrightarrow L^q(\Omega)$ for every $q\in [1,\infty)$.
Moreover,  by standard compact Sobolev embeddings, $K\hookrightarrow H^1_0(\Omega)\hookrightarrow\hookrightarrow L^q(\Omega)$ for every $q\in[1,2^*)$, with $2^*:=2N/(N-2)$ if $N>2$ and $2^*:=+ \infty$ if $N=2$. Hence, the claimed compact embedding can be established via interpolation, following, for instance, the argument in \cite[Proposition 2.8]{BCNW2024}.
\end{remark}

\begin{remark}\label{rem:equivalent-norms}
In view of the definition in \eqref{eq:normV} and of the continuous embedding (i),  the norms $\|\cdot\|_{H^1_\lambda}$ and  $\|\cdot\|_V$ are equivalent in $K$.  As a consequence,  $K$ is closed also with respect to the norm $\|\cdot\|_V$ and, being convex, it is weakly closed both in $H^1_0(\Omega)$ and in $V$.
\end{remark}

As already mentioned in the Introduction, the proof of Theorem \ref{thm:abstract} relies on the nonsmooth critical point theory developed by Szulkin \cite{S}, which provides critical points for $C^1$-perturbations of convex lower semicontinuous functionals. In our setting,  in order to gain compactness, we consider a functional $J_K$ that coincides with $J$ in $K$ and equals $+\infty$ outside $K$, see \eqref{eq:J_def2} for the precise definition.
We stress that  the choice of the function space $V=H^1_0(\Omega)\cap \mathcal M(\Omega)$ is crucial to ensure the $C^1$-regularity of the nonlinear term $\Phi$ in \eqref{eq:J_def2}, cf. Proposition \ref{prop:J-C1}, whereas $\Phi$ is not even well defined in $H^1_0(\Omega)$.  
Thanks to assumption (i) of Theorem \ref{thm:abstract}, working with the functional $J_K$ in $V$ allows to overcome some compactness issues.
As a counterpart, an extra difficulty arises: while we can prove convergence of Palais-Smale sequences in $H^1_0(\Omega)$, it seems hard to establish such a convergence with respect to the $\mathcal L(\Omega)$-norm. To overcome this technical issue, inspired by \cite{LM,ABT}, we formalize a purely geometric version of the nonsmooth mountain pass theorem, see Theorem \ref{thm:ArcoyaBereanuTorres}, where the existence of a Palais-Smale sequence at the mountain pass level $c$ is guaranteed without compactness hypotheses.
Using the $H^1$-strong convergence of this Palais-Smale sequence to a function $u\in K$, through some delicate estimates we show that $J_K(u)=c$ and that $u$ is a critical point of $J_K$ in the nonsmooth sense.  Finally, assumption (ii) of Theorem \ref{thm:abstract} allows to conclude that $u$ is a solution of \eqref{eq:main}.

Next, we provide an application of Theorem \ref{thm:abstract}: we exhibit a special domain and impose additional conditions on the nonlinearity $f$ to verify the assumptions of Theorem \ref{thm:abstract}, thereby guaranteeing the existence of a solution for the problem under consideration.
More precisely, in order to recover the necessary embeddings and compactness required in (i) of Theorem \ref{thm:abstract}, we require the problem to have suitable symmetry. For simplicity, we consider the radially symmetric domain $\Omega=A$ defined in \eqref{eq:A-def}, with $0<R_0 <R_1\le \infty$ and $N\geq3$. 
Hence, the specific problem that we study is the following
\begin{equation}\label{eq:main-appl}
\begin{cases}
-\Delta u + \lambda u = f(x,u) \qquad &\mbox{in }A\\
u>0 &\mbox{in }A\\
u=0 &\mbox{on }{\partial A}
\end{cases}
\end{equation}
where $f:A\times\mathbb{R}^+\to \mathbb{R}$ satisfies $f(x,0)=0$ and from now on we consider its odd extension $f(x,-s)=-f(x,s)$. We look for axially symmetric solutions satisfying further symmetry and monotonicity assumptions: precisely, letting $x=(x_1,\ldots,x_{N-1},x_N)\in A$, we will work with functions in $H^1_0(A) \cap \mathcal M(A)$ that only depend on the two variables 
\begin{equation}\label{eq:def_r-theta}
\begin{gathered}
r:=|x|=\sqrt{x_1^2+\ldots+x_N^2}\in (R_0,R_1),\\
\theta:=\arctan\left(\frac{|x_N|}{\sqrt{x_1^2+\ldots+x_{N-1}^2}}\right) 
=\arcsin\left(\frac{|x_N|}{r}\right) \in \left[0,\pi/2\right],
\end{gathered}
\end{equation}
namely
\[
u(x) = \mathfrak{u}\left(r,\theta \right) 
\quad \text{for a function } \mathfrak{u}: (R_0,R_1)\times\left[0,\frac{\pi}{2}\right] \to \mathbb{R}.
\]
We observe that these are axially symmetric functions
with respect to the $x_N$-axis, which are symmetric with respect to the hyperplane $x_N = 0$.
In addition to these symmetry assumptions, we impose positivity and a monotonicity property that we collect in the following class of functions
\begin{equation}\label{eq:K-def}
\mathcal K:= \Bigl\{u\in H^1_0(A)\,:\, u=\mathfrak{u}(r,\theta),\: \,u\ge 0,\: \mathfrak{u}_\theta\le 0\:  \mbox{ a.e.}
\Bigr\},
\end{equation}
where $\mathfrak{u}_\theta$ stands for the weak partial derivative of $\mathfrak{u}$ with respect to the variable $\theta$.  
We prove in Proposition \ref{prop:cont-embedding} that $\mathcal K\subset\mathcal M(A)$ with continuous embedding.

In order to ensure that the pointwise invariance (ii) of Theorem \ref{thm:abstract} holds in $\mathcal K$, we need to require additional assumptions on $f$, namely:
\begin{properties}{f}
\setcounter{enumi}{4}
\item \label{f6} for all $x\in A$ and $s\ge 0$, $f(x,s)=\mathfrak{f}(r,\theta, s)$ ;\smallskip
\item \label{f7} for all $s\ge 0$, $f(\cdot,s)\in W^{1,1}_{\mathrm{loc}}(A)$ and $\mathfrak{f}_\theta\le 0$ a.e.; \smallskip
\item \label{f8} for all $x\in A$, $f(x,\cdot)$ is differentiable and $\partial _s f(x,s)\ge 0 $ for every $s\ge0$;\smallskip
\item \label{f10} 
for all $M>0$, there exist $d_M\in L^2(A)$ and $D_M>0$ such that $|\nabla _xf(x,s)|\le d_M(x)$ and $|\partial_s f(x,s)|\le D_M$ for every $x\in A$ and $s\in[0,M]$.
 \end{properties}
 
We show that, under the previous assumptions, Theorem \ref{thm:abstract} applies, thus providing a mountain pass solution of \eqref{eq:main-appl} belonging to the cone $\mathcal{K}$.

\begin{theorem}\label{thm:applied1}
Let $A$ be as in \eqref{eq:A-def}, with $0<R_0 <R_1\le \infty$ and $N\geq3$.
Let $\lambda> 0$ ($\lambda \geq0$ if $A$ is bounded) and $f$ satisfy \ref{f1}--\ref{f5}, with $\Omega=A$, and \ref{f6}--\ref{f10}. 
Then there exists a nontrivial solution of \eqref{eq:main-appl} belonging to the set $\mathcal K$ defined in \eqref{eq:K-def} at energy level
\begin{equation}\label{eq:c_level_def}
c=\inf_{\gamma\in \Upsilon}\sup_{t\in [0,1]}J(\gamma(t)),
\end{equation}
with $\Upsilon =\{\gamma\in C([0,1];\mathcal K)\,:\,\gamma(0)=0\neq\gamma(1),\, J(\gamma(1))\le 0\}$ and $J$ as in \eqref{eq:J-def} with $\Omega=A$.
\end{theorem}

In particular, Theorem \ref{thm:applied1} applies to exponential-type nonlinearities of the form
\begin{equation}\label{eq:prototype}
\begin{gathered}
f(x,s)=w(x)|s|^{\beta-2} s \exp_m \left( |s|^\beta\right)\\[1.3ex]
\text{with }\beta\in (0,2) \text{ and } m\in \N\setminus\{0\}
\text{ satisfying } \beta(m+1)>2,
\end{gathered}
\end{equation}
where
\begin{equation} \label{eq:expm}
\exp_m(s):=e^s-\sum_{i=0}^{m-1} \frac{s^i}{i!}
\end{equation}
and the weight $w$ satisfies 
\begin{equation}\label{eq:hp-w}
\begin{gathered}
w\in W^{1,1}_{\mathrm{loc}}(A)\cap C(A)\cap L^\infty(A),\quad |\nabla w|\in L^2(A),\\ 
w = \mathfrak{w}(r,\theta) > 0\mbox{ and } \mathfrak{w}_\theta \le 0\quad\mbox{in } A.
\end{gathered}
\end{equation}
Notice that when $\beta\in(0,1)$ we intend $f(x,\cdot)$ extended by continuity at $s=0$.
We further observe that \eqref{eq:prototype} reduces to the nonlinearity in problem \eqref{eq:main-exp} when $w\equiv1$,  $\beta\in(1,2)$, and $m=1$.

Theorem \ref{thm:applied1} also covers possibly inhomogeneous power-type nonlinearities of the form
\begin{equation}\label{eq:prototype-power}
f(x,s)=w(x)\left(|s|^{p-2}s + |s|^{\mathfrak{p}-2}s \right)
\quad\mbox{with }2<p\leq \mathfrak{p},
\end{equation}
where $w$ satisfies \eqref{eq:hp-w}.
When $p=\mathfrak{p}$, we thus recover the result established in \cite[Theorem 1.1]{BCNW2023} for pure power-type nonlinearities (see also \cite{BCNW2024} for the case of the exterior of a ball $R_1=\infty$). Both claims are justified in Remark \ref{rem:prototype}.

We observe that, since the extra radial symmetry of the problem is not required in the abstract Theorem \ref{thm:abstract}, the existence Theorem \ref{thm:applied1} can be extended to more general domains as in \cite[Section 4]{BCNW2024}. 
On the opposite side, when both the domain and the nonlinearity are radial, namely $f$ satisfies
\begin{properties}{f'}
\setcounter{enumi}{4}
\item \label{f'6} for all $x\in A$ and $s\ge 0$, $f(x,s)=\mathfrak{f}(r, s)$,
\end{properties}
the existence of a radial solution can be obtained via standard techniques. 
In this context, we prove that the solution found in Theorem \ref{thm:applied1} breaks the symmetry, provided that the dimension $N$ or the parameter $\lambda$ or the inner radius $R_0$ are sufficiently large.

Our symmetry breaking result is the following.

\begin{theorem}\label{thm_nonradial}
Let $A$ be as in \eqref{eq:A-def}, with $0<R_0 <R_1\le \infty$ and $N\geq3$.
Let $\lambda> 0$ ($\lambda \geq0$ if $A$ is bounded) and $f$ satisfy \ref{f1}--\ref{f5} with $\Omega=A$, \ref{f'6}, and \ref{f7}--\ref{f10}. Suppose moreover that 
\begin{properties}{f}
\setcounter{enumi}{8}
\item \label{f11} for all $x\in A$, $f(x,\cdot)\in C^1(\mathbb R^+)$ and for all $\alpha>0$ there exists $C=C(\alpha)>0$ such that, for every $x\in A$, $s>0$, \smallskip
\begin{itemize}
\item[(b)] \emph{if $\Omega$ is bounded}: $|\partial_s f(x,s)|\le C + (e^{\alpha s^2} - 1)$ \smallskip
\item[(u)] \emph{if $\Omega$ is unbounded}: $|\partial_s f(x,s)|\le C|s|^{\mu-1} + (e^{\alpha s^2} - 1)$ \smallskip
\end{itemize}  	
with $\mu>1$ as in \ref{f5}-{\rm (u)}; \smallskip
\item \label{f12} letting $H=H(N,\lambda,R_0)>0$ as in \eqref{eq:def-H} and
\[
\delta > \frac{2N}{H}+1\;\mbox{ if }\lambda=0,
\quad  
\delta = \frac{2N}{H}+1\; \mbox{ if }\lambda>0,
\]
 it holds 
$$\partial_sf(x,s) s \ge \delta f(x,s) \quad  \text{ for all } x\in A \text{ and }s> 0 . $$
\end{properties}
Then problem \eqref{eq:main-appl} has a nonradial solution belonging to the set $\mathcal K$. More precisely, any solution of \eqref{eq:main-appl} belonging to $\mathcal K$, at energy level $c$ defined in \eqref{eq:c_level_def}, is nonradial.
\end{theorem}
The proof of Theorem \ref{thm_nonradial} proceeds by contradiction. Assuming that the mountain pass solution obtained in Theorem \ref{thm:applied1} is radial allows us to construct an admissible path along which the maximum value of the energy functional is below the mountain pass level, leading to a contradiction.

By specializing Theorem \ref{thm_nonradial} to the explicit nonlinearities \eqref{eq:prototype} and \eqref{eq:prototype-power} with a radial weight $w = \mathfrak{w}(r)$, we obtain the following symmetry breaking result for both exponential- and power-type nonlinearities.

\begin{corollary}\label{cor}
Let $A$ be as in \eqref{eq:A-def}, with $0<R_0 <R_1\le \infty$ and $N\geq3$.
Let $\lambda> 0$ ($\lambda \geq0$ if $A$ is bounded) and let $w = \mathfrak{w}(r)$ be a radial weight satisfying \eqref{eq:hp-w}.
\begin{itemize}
\item[(i)] Let $f$ be as in \eqref{eq:prototype}.  If
\[
\beta(m+1) > 2+\frac{2N}{H} \mbox{ if }\lambda=0,
\qquad
\beta(m+1) \ge 2+\frac{2N}{H} \mbox{ if }\lambda>0,
\]
with $H$ as in \eqref{eq:def-H}, then \eqref{eq:main-appl} has a nonradial solution belonging to $\mathcal K$. 
\item[(ii)] Let $f$ be as in \eqref{eq:prototype-power}.  If
\[
p > 2+\frac{2N}{H} \mbox{ if }\lambda=0,
\qquad
p \geq 2+\frac{2N}{H} \mbox{ if }\lambda>0,
\]
with $H$ as in \eqref{eq:def-H}, then \eqref{eq:main-appl} has a nonradial solution belonging to $\mathcal K$. 
\end{itemize}
\end{corollary}

Notice that Corollary \ref{cor} (i) reduces to Theorem \ref{thm:main-exp} when $w\equiv 1$, $\beta\in(1,2)$, and $m=1$, while case (ii) when $p=\mathfrak{p}$ recovers \cite[Theorem 1.1]{BCNW2023} (see also \cite{BCNW2024} when $R_1=\infty$). 

The paper is structured as follows. 
In Section \ref{sec:abstract_theorem} we recall some properties of the Orlicz spaces $\mathcal{L}(\Omega)$, $\mathcal{M}(\Omega)$ and prove,  via Szulkin's theory, the abstract existence result Theorem \ref{thm:abstract}. Section \ref{sec:proof-applied-thm} is devoted to its application, namely Theorem \ref{thm:applied1}, which in turn requires the proof of a Pohozaev-Trudinger-Moser inequality in the cone $\mathcal K$, see Proposition \ref{cor:Trud-Moser1}.
Finally, in Section \ref{sec:case-radial}, we establish the occurrence of symmetry breaking in the radial setting, thereby proving Theorem \ref{thm_nonradial}.
As a particular case, this also leads to Corollary \ref{cor} and Theorem \ref{thm:main-exp}.

\section{Abstract result}\label{sec:abstract_theorem}
\subsection{Preliminaries on Orlicz spaces}
We start by recalling the known property that $\mathcal L(\Omega)$ is continuously embedded in $L^q(\Omega)$ for every $q\in[2,\infty)$. 

\begin{lemma}[{\cite[Lemma 11]{RufTerraneo}}]\label{lem:L-embedded}
For every $q\in [2,\infty)$ ($q\in [1,\infty)$ if $\Omega$ is bounded) there exists a constant $C=C(q)>0$ such that
\begin{equation}\label{eq:L-embedded}
\|u\|_{L^q(\Omega)} \leq C \|u\|_{\mathcal{L}(\Omega)}
\quad\text{for every } u\in\mathcal{L}(\Omega).
\end{equation}
\end{lemma}

It is well known that the $\mathcal{L}(\Omega)$-norm convergence implies the convergence in mean, see \cite[\S 8.13]{AF}.  For the sake of clarity,  we report here the proof of this property, together with a useful further consequence.

\begin{lemma}\label{lem:auxiliary}
Let $u\in \mathcal{M}(\Omega)$ and $(u_n)\subset \mathcal{M}(\Omega)$ be such that
\[
\lim_{n\to\infty} \|u_n-u\|_{\mathcal{L}(\Omega)} =0.
\]
Then, for every $\alpha>0$,
\begin{equation}\label{eq:conv-mean}
\lim_{n\to\infty} \int_\Omega \left( e^{\alpha (u_n-u)^2} -1 \right)\,dx=0.
\end{equation}
This in turn further implies that for every $\alpha>0$,
\begin{equation}\label{eq:conv-mod}
\lim_{n\to\infty} \int_\Omega \left( e^{\alpha u_n^2} -1 \right)\,dx = \int_\Omega \left( e^{\alpha u^2} -1 \right)\,dx.
\end{equation}
\end{lemma}
\begin{proof}
Let us recall the known inequality
\[
e^{s^2}-1 \leq \varepsilon \left( e^{\left(\frac{s}{\varepsilon}\right)^2} -1 \right)
\quad \text{for all }\varepsilon\in(0,1]\text{ and }s\geq0,
\]
which can be easily derived expanding in series the exponential functions.
As, for every $\alpha>0$,  $\sqrt{\alpha}\|u_n-u\|_{\mathcal{L}(\Omega)}<1$ for $n$ large, we can apply the previous inequality with $\varepsilon=\sqrt{\alpha}\|u_n-u\|_{\mathcal{L}(\Omega)}$, to obtain, for $n$ large enough,
\[
\begin{aligned}
\int_\Omega \left( e^{\alpha (u_n-u)^2} -1 \right)\,dx
&\leq \sqrt{\alpha}\|u_n-u\|_{\mathcal{L}(\Omega)} \int_\Omega \left( e^{\frac{\alpha (u_n-u)^2}{\alpha\|u_n-u\|^2_{\mathcal{L}(\Omega)}}} -1 \right)\,dx \\
&\leq \sqrt{\alpha}\|u_n-u\|_{\mathcal{L}(\Omega)},
\end{aligned}
\]
where in the last line we have used the definition of $\mathcal{L}(\Omega)$-norm. Then \eqref{eq:conv-mean} follows. As for \eqref{eq:conv-mod}, we observe that 
$e^{\alpha u_n^2} -1\to e^{\alpha u^2} -1$ a.e. in $\Omega$, moreover
\[
e^{\alpha u_n^2} -1\le e^{2\alpha (u_n-u)^2}e^{2\alpha u^2} -1 \le \frac{e^{4\alpha (u_n-u)^2}-1}{2}+\frac{e^{4\alpha u^2} -1}{2}=:g_n.
\]
Set $g:=\frac{e^{4\alpha u^2} -1}{2}$ and observe that $g_n,\,g\in L^1(\Omega)$, since $u_n,\,u\in \mathcal M(\Omega)$, and that $g_n\to g$ a.e. in $\Omega$. By \eqref{eq:conv-mean}, it also holds that $\int_\Omega g_n\to \int_\Omega g$, hence, by the General Dominated Convergence Theorem (see for example\cite[Section 4.4]{Royden}), \eqref{eq:conv-mod} is proved. 
\end{proof}

\subsection{Properties of the nonlinear term}
We collect here some useful observations and inequalities that descend from the assumptions on $f$. We begin by commenting on the different hypotheses in the bounded and unbounded case. 

\begin{remark}\label{rem:f5u-b}
The second condition in \ref{f3}-(u), that is,
\begin{equation}
\label{f4new}
F(x,s)\ge c_1(x)|s|^\sigma-c_2(x)\quad	
\mbox{ for all }s\in \mathbb R\mbox{ and } x\in\Omega,
\end{equation}
also holds when $\Omega$ is bounded, with constant $c_1>0$ and $c_2\in\mathbb{R}$. Indeed, in this case, it descends from \ref{f3}-(b) and \ref{f1}.
Hence, in what follows, we will refer to \eqref{f4new} for both bounded and unbounded domains.
Condition \ref{f5}-(u), required for the unbounded case, is stronger than the corresponding condition \ref{f5}-(b) for the bounded case. Indeed,  suppose that $\Omega$ is bounded and that \ref{f5}-(u) holds, then there exist $\mu>1$ and $C>0$ such that for every $x\in\Omega$
\begin{equation}\label{eq:fsmu}
|f(x,s)|\le C|s|^{\mu}\quad\mbox{ for $|s|$ small}.
\end{equation}
Then, being $\mu>1$,
\[
\limsup_{s\to 0}\frac{2F(x,s)}{s^2}=\limsup_{s\to 0}\frac{2\int_0^s f(x,\tau)d\tau}{s^2}\le \limsup_{s\to 0}\frac{2C}{\mu+1}|s|^{\mu-1}= 0 <\lambda_1+\lambda,
\]
with $C>0$ as in \eqref{eq:fsmu}, so that \ref{f5}-(b) holds.

Finally, we notice that if $\lim_{s\to 0}|f(x,s)|/|s|^{\sigma-1}$ exists,  \ref{f3}-(u) implies \ref{f5}-(u) with $\mu=\sigma-1$ by L'H\^opital's rule.
\end{remark}

\medskip
Next, we derive some consequences of assumptions \ref{f1}-\ref{f5} that will be useful in the remaining part of the section.

We first observe that, by \ref{f1}, 
\[
F(x,s)=\int_{\min\{0,s\}}^{\max\{0,s\}}|f(x,\tau)|d\tau\ge 0\quad\mbox{for all $(x,s)\in\Omega\times\mathbb R$.}
\]

Moreover, for every $\alpha>0$ and $\tilde{\mu}\in[1,\mu]$ the following inequalities hold
\begin{equation}\label{eq:asterisco_mu_nu}
\begin{aligned}
\mbox{if $\Omega$ is bounded:}\quad &|f(x,s)|\le C+|s|^{\nu}(e^{\alpha s^2}-1) \\
\mbox{if $\Omega$ is unbounded:}\quad &|f(x,s)|\le C|s|^{\tilde\mu}+|s|^{\nu}(e^{\alpha s^2}-1),
\end{aligned}
\end{equation}
for all $(x,s)\in \Omega\times\mathbb R$ and $\nu\ge0$, with $C=C(\alpha)>0$ in the first line and $C=C(\alpha, \tilde{\mu})>0$ in the second line.

Indeed, by \ref{f2}, for every $\alpha>0$ there exists $\widehat M =\widehat M(\alpha)>1$ such that 
\begin{equation}\label{eq:at_infinity}
|f(x,s)|\le e^{\alpha s^2}-1\le |s|^{\nu}(e^{\alpha s^2}-1) \quad\mbox{ for all $x\in\Omega$, $|s|\ge \widehat M$, and $\nu\ge 0$}.
\end{equation}
Therefore, combining the previous estimate with \ref{f1}, we have that for every $\alpha>0$ there exists $C=C(\alpha)>0$ such that 
\[
|f(x,s)|\le C+|s|^{\nu}(e^{\alpha s^2}-1)\quad\mbox{for all $(x,s)\in\Omega\times\mathbb R$ and $\nu\ge 0$}.
\]

If $\Omega$ is unbounded, the above bound will not be helpful for our purposes. However, in this case, by \ref{f5}-(u), there exists $\delta>0$ so small that  
\begin{equation*}\label{eq:near_zero}
|f(x,s)|\le C |s|^{\mu}\le C |s|^{\tilde{\mu}}
\quad\mbox{ for all $x\in \Omega$, $|s|\le \delta$ and $\tilde \mu\in [1,\mu]$}
\end{equation*}
and by \ref{f1} there exists $C=C(\alpha,\tilde\mu)>0$ such that
\begin{equation*}\label{eq:in_the_middle}
|f(x,s)|\le C_{\widehat M} \frac{\delta^{\tilde{\mu}}}{\delta^{\tilde{\mu}}}\le C |s|^{\tilde{\mu}} \quad\mbox{for all $x\in\Omega$ and  $\delta<|s|<\widehat M$.}
\end{equation*}
Combining the previous two estimates with \eqref{eq:at_infinity}, in the unbounded case, 
we get that for every $\alpha>0$ and $\tilde{\mu}\in[1,\mu]$ there exists $C=C(\alpha,\tilde{\mu})>0$ such that 
\eqref{eq:asterisco_mu_nu} holds for all $(x,s)\in \Omega\times\mathbb R$ and $\nu\ge0$.

Finally, for all $\alpha>0$, $\tilde{\mu}\in [1,\mu]$, $\nu\ge 0$, and $\varepsilon>0$ sufficiently small 
\begin{equation}\label{eq:conseq-hyp-on-f}
\begin{aligned}
\mbox{if $\Omega$ is bounded:}&\quad F(x,s)\le \frac{1-\varepsilon}{2}(\lambda_1+\lambda)s^2+C|s|^{\nu+1}(e^{\alpha s^2}-1)\\
\mbox{if $\Omega$ is unbounded:}&\quad F(x,s)\le C|s|^{\tilde{\mu}+1}+|s|^{\nu+1}(e^{\alpha s^2}-1),
\end{aligned}
\end{equation}
for all $(x,s)\in \Omega\times\mathbb R$, with $C=C(\alpha,\nu)>0$ in the bounded case and $C=C(\alpha,\tilde{\mu})>0$ in the unbounded case.

Indeed, in the unbounded case, \eqref{eq:conseq-hyp-on-f} is obtained multiplying \eqref{eq:asterisco_mu_nu} by $|s|$ and using \ref{f3}-(u).
In the bounded case, for every $\alpha>0$ and $\nu\ge 0$,
multiplying \eqref{eq:at_infinity} by $|s|\ge \max\{\widehat M,\overline{M}\}=:\widetilde{M}=\widetilde{M}(\alpha)$ and using \ref{f3}-(b), we obtain 
\[
F(x,s)\le \frac1\sigma |s|^{\nu+1}(e^{\alpha s^2}-1)\quad\mbox{ for all $x\in\Omega$ and $|s|\ge \widetilde{M}$.} 
\]
By \ref{f5}-(b) there exist $\delta>0$ and $\varepsilon \in(0,1)$ such that 
\begin{equation*}\label{eq:Fnearzero}
F(x,s)\le \frac{1-\varepsilon}{2}(\lambda_1+\lambda)s^2\quad\mbox{for all $x\in \Omega$ and $|s|\le \delta$}.
\end{equation*} 
Moreover, by \ref{f1}, we have for all $x\in\Omega$ and  $\delta<|s|<\widetilde M$
\begin{equation*}\label{eq:in_the_middle-b}
F(x,s)=\int_{\min\{0,s\}}^{\max\{0,s\}}|f(x,\tau)|d\tau\le C_{\widetilde M} |s|\le \frac{C_{\widetilde M} \widetilde M}{\delta^{\nu+1}(e^{\alpha \delta^2}-1)}|s|^{\nu+1}(e^{\alpha s^2}-1).
\end{equation*}
Combining together the last three estimates, and recalling that $\sigma>2$, we get \eqref{eq:conseq-hyp-on-f} for $\Omega$ bounded.

The following convergence result will ensure the $C^1$-regularity of the nonlinear term in the energy functional.

\begin{lemma}\label{lem:convergenceOrlicz}
Under the assumptions of Theorem \ref{thm:abstract}, let $(u_n)\subset \mathcal M(\Omega)$ be a sequence such that $u_n \to u$ in $\mathcal L(\Omega)$ for some $u\in \mathcal M(\Omega)$. Then 
\[
\lim_{n\to\infty}\int_\Omega(f(x,u_n)-f(x,u))^2 \, dx= 0.
\]
\end{lemma}

\begin{proof}
We will prove the stronger thesis
\begin{equation}\label{eq:strongerth}
\lim_{n\to\infty}\|f(\cdot ,u_n)-f(\cdot,u)\|_{L^q(\Omega)}= 0\, \mbox{ if }
q \ge 
\begin{cases} 
1&\mbox{if $\Omega$ is bounded}\\
\max\left\{ 1,\frac{2}{\mu} \right\} &\mbox{ if $\Omega$ is unbounded,}
\end{cases}
\end{equation}
with $\mu>1$ as in \ref{f5}-(u).

Since $\|u_n-u\|_{\mathcal L(\Omega)}\to 0$, Lemma \ref{lem:L-embedded} yields that
$u_n\to u$  a.e. in $\Omega$ up to a subsequence. Moreover, by \ref{f1},  $f(x,u_n)\to f(x,u)$ a.e. in $\Omega$.
Now, if $\Omega$ is unbounded, for $q$ as in \eqref{eq:strongerth}, Lemma \ref{lem:L-embedded} yields also that 
\begin{equation}\label{eq:untouinLqmu}
u_n \to u \quad\mbox{ in }L^{q\mu}(\Omega) 
\end{equation}
and, in virtue of \eqref{eq:asterisco_mu_nu} with $\alpha>0$, $\tilde\mu=\mu$ and $\nu=0$,  we have
\[
\begin{aligned}
|f(x,u_n)-f(x,u)|^q
&\leq 2^{q-1}\left[ \: |f(x,u_n)|^q+|f(x,u)|^q \: \right]\\
&\leq 2^{q-1} \left[(C|u_n|^\mu+ e^{\alpha u_n^2}-1)^q+(C|u|^\mu+e^{\alpha u^2}-1)^q\right]\\
&\leq \bar C \left[|u_n|^{q\mu}+ |u|^{q\mu} + (e^{q\alpha u_n^2}-1)+(e^{q\alpha u^2}-1)\right] =: g_n,\\
\end{aligned}
\]
where we have used the convexity of the function $|\cdot|^q$ and the inequality
\begin{equation}\label{eq:favourite_ineq}
(e^s-1)^t \le e^{t s}-1\quad\mbox{ for every $s\ge0$ and $t\ge 1$.}
\end{equation} 
Notice that the positive constant $\bar C$ is in the definition of $g_n$ is independent of $n$. Let
\[
g:=2 \bar C \left[|u|^{q \mu} + (e^{q\alpha u^2}-1) \right].
\]
We have that $g, g_n \in L^1(\Omega)$ for every $n$; moreover, by the previous considerations, $g_n\to g$ a.e. in $\Omega$. In addition, by \eqref{eq:untouinLqmu} and \eqref{eq:conv-mod}, $\int_\Omega g_n\to\int_\Omega g$. Arguing similarly for $\Omega$ bounded, we have for $q\ge1$
\[
\begin{aligned}
|f(x,u_n)-f(x,u)|^q 
&\leq \bar C \left[2 + (e^{q\alpha u_n^2}-1) + (e^{q\alpha u^2}-1)\right]=: h_n,
\end{aligned}
\]
where again $\bar C$ denotes a positive constant independent of $n$, and setting
\[
h:=2\bar C e^{q\alpha u^2},
\]
we have that $h, h_n \in L^1(\Omega)$ for every $n$ and that $h_n\to h$ a.e. in $\Omega$ and $\int_\Omega h_n\to \int_\Omega h$.
As a consequence, in both cases, the General Dominated Convergence Theorem applies, thus providing   the stronger thesis.
\end{proof}

In the following lemma we collect some convergence results for the nonlinear term that will be useful in the subsequent part, mainly to prove compactness.

\begin{lemma}\label{lem:convergences}
Under the assumptions of Theorem \ref{thm:abstract}, let $(u_n)\subset K$ be a sequence such that $u_n\rightharpoonup u$ in $H^1_0(\Omega)$ for some $u\in K$. Then 
\begin{equation}\label{eq:boundedness-f}
(f(x,u_n)) \mbox{ is bounded in $L^q(\Omega)$ if }
q\ge 
\begin{cases} 
1&\mbox{if $\Omega$ is bounded}\\
\max\left\{1,\frac{2}{\mu}\right\} &\mbox{if $\Omega$ is unbounded.}
\end{cases}
\end{equation}
As a consequence, the following limits hold true: 
\begin{equation}\label{eq:fun_to_0}
\lim_{n\to \infty}\int_\Omega f(x,u_n)(u_n-u)\, dx= 0;
\end{equation}
\begin{equation}\label{eq:fun-to-fu_Lp}
\lim_{n\to\infty}\int_\Omega(f(x,u_n)-f(x,u))^2 \, dx= 0;
\end{equation} 
\begin{equation}\label{eq:convergence_fun-un}
\lim_{n\to\infty}\int_\Omega f(x,u_n)u_n \, dx = \int_\Omega f(x,u)u\,dx.
\end{equation}
\end{lemma}

\begin{proof}
We first show the boundedness of $(f(x,u_n))$ in $L^q(\Omega)$ in the case $\Omega$ unbounded. 
Let $q\geq \max\{1,2/\mu\}$. 
Using \eqref{eq:asterisco_mu_nu} with $\tilde{\mu}=\mu$ and $\nu=0$, we get, for every $\alpha>0$,
\begin{equation}
\label{eq:inequalities-for-f-bdd}
\begin{aligned}
\int_\Omega |f(x,u_n)|^q \, dx &\le 2^{q-1} \int_\Omega \left[C^q|u_n|^{q\mu} + (e^{\alpha u_n^2}-1)^q\right] \, dx\\
&\le 2^{q-1} \left[C^q\|u_n\|_{L^{q\mu}(\Omega)}^{q\mu} + \int_\Omega	(e^{q\alpha u_n^2}-1)\, dx\right],
\end{aligned}
\end{equation}
where we have used the convexity of the function $|\cdot|^q$ and the inequality \eqref{eq:favourite_ineq}.
Moreover, since $(u_n)$ is bounded in $V$ in view of the equivalence of the norms in $K$, see Remark \ref{rem:equivalent-norms}, $\|u_n\|_{\mathcal L(\Omega)}\le \tilde c$ for some $\tilde c>0$. Thus, by the definition of the Luxemburg norm $\|\cdot\|_{\mathcal L(\Omega)}$,
\[
\int_\Omega \left(e^{(u_n/\tilde c)^2}-1\right)\, dx\le 1. 
\]
Therefore, taking $\alpha$ such that $q\alpha	\le 1/\tilde c^2$, we obtain
\begin{equation*}
\sup_{n\in\mathbb N}\int_\Omega (e^{q\alpha u_n^2}-1)\, dx<+\infty.
\end{equation*}
Then, taking also into account the embedding of $\mathcal M(\Omega)$ in $L^{q\mu}(\Omega)$ given by Lemma \ref{lem:L-embedded}, by \eqref{eq:inequalities-for-f-bdd} we can conclude the proof in this case.

The proof for the bounded case is analogous, the only difference being that here one uses the corresponding version of \eqref{eq:asterisco_mu_nu}, that constants are integrable, and that $\mathcal M(\Omega)\hookrightarrow L^q(\Omega)$ for all $q\ge 1$.

Next we prove \eqref{eq:fun_to_0}--\eqref{eq:convergence_fun-un}.

{\it Proof of \eqref{eq:fun_to_0}}: Let $\Omega$ be unbounded and $q\in (\max\{1,2/\mu\},2)$. Then the cone $K$ is compactly embedded in $L^{q'}(\Omega)$ by assumption (i) of Theorem \ref{thm:abstract}, with $q'>2$ being the conjugate exponent of $q$. 
By H\"older's inequality and the boundedness of $(f(x,u_n))$ in $L^q(\Omega)$ proved in \eqref{eq:boundedness-f},
\[
\begin{aligned}
\left|\int_\Omega f(x,u_n)(u_n-u)\, dx\right| &\le \left(\int_\Omega |f(x,u_n)|^q \, dx\right)^{1/q}\left(\int_\Omega|u_n-u|^{q'}\, dx\right)^{1/q'}\\
 &\le C\|u_n-u\|_{L^{q'}(\Omega)} \to 0\quad\mbox{as }n\to \infty.
 \end{aligned}
\]
For $\Omega$ bounded, one can argue similarly taking $q=q'=2$ and using the compact embedding of $H^1_0(\Omega)$ in $L^2(\Omega)$.

{\it Proof of \eqref{eq:fun-to-fu_Lp}}: We will prove the stronger thesis
\begin{equation}\label{eq:strongerth2}
\lim_{n\to\infty}\|f(\cdot,u_n)-f(\cdot,u)\|_{L^q(\Omega)}= 0\, \mbox{ if }
q 
\begin{cases} 
\ge 1&\mbox{if $\Omega$ is bounded}\\
>\max\left\{ 1,\frac{2}{\mu} \right\} &\mbox{ if $\Omega$ is unbounded.}
\end{cases}
\end{equation}

Since $u_n\to u$ a.e. for a renamed subsequence and $f$ is continuous by \ref{f1}, $f(x,u_n)\to f(x,u)$ a.e. in $\Omega$. By a consequence of Brezis-Lieb's Lemma \cite{BL} it is enough to prove that $\|f(\cdot,u_n)\|_{L^q(\Omega)}\to \|f(\cdot,u)\|_{L^q(\Omega)}$.
By Proposition \ref{prop:zero-measure}, there exists a strictly increasing sequence $(M_k)\subset\mathbb R$ such that $M_k\to \infty$ and $|\{u=M_k\}|=0$ for every $k\in \mathbb N$. Now, let $k\in\mathbb N$ be fixed. We have for $q$ as in \eqref{eq:strongerth2} and for every $n\in\mathbb N$
\begin{equation}\label{eq:f-splitting}
\begin{aligned}
&\left|\int_\Omega\left(|f(x,u_n)|^q-|f(x,u)|^q\right)\, dx\right|\\
 &\qquad\quad\le \int_{\{|u_n|\ge M_k\}}|f(x,u_n)|^q \, dx\\
&\qquad\qquad + \left|\int_\Omega\left(|f(x,u_n)|^q \chi_{\{|u_n|< M_k\}}-|f(x,u)|^q \chi_{\{|u|< M_k\}}\right)\, dx\right|\\
&\qquad\qquad + \int_{\{|u|\ge M_k\}}|f(x,u)|^q \, dx =: I_{n,M_k}+I'_{n,M_k}+I''_{M_k},
\end{aligned}
\end{equation}
where, for a set $D\subseteq \R^N$, $\chi_D$ denotes the characteristic function of $D$, i.e. 
\begin{equation}
\label{eq:chi-def}
\chi_D(x)=\begin{cases}
1 \quad & \text{if }x\in D \\
0 & \text{otherwise.}
\end{cases}
\end{equation}
As for $I_{n,M_k}$,  by \eqref{eq:boundedness-f} we can write 
\[
\begin{aligned}I_{n,M_k}&\le \frac{1}{M_k}\int_{\{|u_n|\ge M_k\}}|f(x,u_n)|^q |u_n|\, dx\\&
\le \frac{1}{M_k}\|u_n\|_{L^2(\Omega)}\left(\int_\Omega|f(x,u_n)|^{2q}\, dx\right)^{\frac{1}{2}}\le \frac{C}{M_k},
\end{aligned}
\]
for some $C>0$ independent of $n$ and $M_k$.
Similarly, $I''_{M_k}\le C/M_k$. In order to estimate $I'_{n,M_k}$, we first observe that 
\[|f(x,u_n)|^q\chi_{\{|u_n|<M_k\}}\to |f(x,u)|^q\chi_{\{|u| < M_k\}}\quad\mbox{a.e. in $\Omega$, as $n\to\infty$,}
\]
by Proposition \ref{prop:chi}.

Next, we distinguish two cases, depending on whether $\Omega$ is bounded or unbounded. Let $q$ be as in \eqref{eq:strongerth2}.
If $\Omega$ is unbounded, using \eqref{eq:asterisco_mu_nu} with $\alpha>0$, $\tilde\mu=\mu$ and $\nu=1$, the inequality \eqref{eq:favourite_ineq}, and 
\begin{equation}
\label{eq:second-favourite-ineq}
e^s-1\le s e^s\quad\mbox{for all }s\ge 0,
\end{equation}
we have 
\[
\begin{aligned}
|f(x,u_n)|^q\chi_{\{|u_n|<M_k\}}&\le 2^{q-1} \left[C^q|u_n|^{q\mu} + |u_n|^q(e^{q\alpha u_n^2}-1)\right]\cdot\chi_{\{|u_n|<M_k\}}\\ 
&\le 2^{q-1}C^q |u_n|^{q\mu}+ 2^{q-1}q\alpha e^{q\alpha M_k^2}|u_n|^{q+2}\\
&= \kappa_1 |u_n|^{q\mu}+ \kappa^{(k)}_2|u_n|^{q+2}=:g^{(k)}_n,
\end{aligned}
\] 
where $\kappa_1,\,\kappa^{(k)}_2>0$ are independent of $n$.
Moreover, setting $g^{(k)}:=\kappa_1 |u|^{q\mu}+ \kappa^{(k)}_2|u|^{q+2}$, we have that $g^{(k)}_n,\,g^{(k)}\in L^1(\Omega)$ and that $g^{(k)}_n\to g^{(k)}$ a.e. in $\Omega$, as $n\to\infty$. By condition (i) of Theorem \ref{thm:abstract}, $K$ is compactly embedded in $L^{q\mu}(\Omega)$ and in $L^{q+2}(\Omega)$, then $\int_\Omega g^{(k)}_n\to \int_\Omega g^{(k)}$ as $n\to\infty$. 
Similarly, if $\Omega$ is bounded, using \eqref{eq:asterisco_mu_nu} we obtain
\[
\begin{aligned}
|f(x,u_n)|^q\chi_{\{|u_n|<M_k\}}&\le 2^{q-1} \left[C^q + |u_n|^q(e^{q\alpha u_n^2}-1)\right]\cdot\chi_{\{|u_n|<M_k\}}\\ 
&\le 2^{q-1}C^q + 2^{q-1}q\alpha e^{q\alpha M_k^2}|u_n|^{q+2}= \kappa_3+ \kappa^{(k)}_4|u_n|^{q+2}=:h^{(k)}_n,
\end{aligned}
\] 
where $\kappa_3,\,\kappa^{(k)}_4>0$ are independent of $n$. Moreover, setting $h^{(k)}:=\kappa_3 + \kappa^{(k)}_4|u|^{q+2}$, we have that $h^{(k)}_n,\,h^{(k)}\in L^1(\Omega)$ and that $h^{(k)}_n\to h^{(k)}$ a.e. in $\Omega$ and $\int_\Omega h^{(k)}_n\to\int_\Omega h^{(k)}$ as $n\to\infty$. 
In both cases, by the General Dominated Convergence Theorem, we have that 
$I'_{n,M_k}\to 0$ as $n\to \infty$, for every $k\in \mathbb N$. Altogether, for every $k\in\mathbb N$
\[
\limsup_{n\to\infty}(I_{n,M_k}+I'_{n,M_k}+I''_{M_k})\le \frac{C}{M_k}.
\]
Therefore, by \eqref{eq:f-splitting},
\[
\begin{aligned}
0\le \liminf_{n\to\infty}&\left|\int_\Omega\left(|f(x,u_n)|^q-|f(x,u)|^q\right)\, dx\right|\\
&\le \limsup_{n\to\infty}\left|\int_\Omega\left(|f(x,u_n)|^q-|f(x,u)|^q\right)\, dx\right| \le \frac{C}{M_k}\to 0 \quad\mbox{as }k\to\infty,
\end{aligned}
\]
which concludes the proof of \eqref{eq:strongerth2}.

{\it Proof of \eqref{eq:convergence_fun-un}}: By \eqref{eq:fun_to_0}, the Cauchy-Schwarz inequality, and \eqref{eq:fun-to-fu_Lp}
\[
\begin{aligned}
&\left|\int_\Omega f(x,u_n)u_n \, dx - \int_\Omega f(x,u)u\, \, dx\right|\\
&\qquad\le \left|\int_\Omega f(x,u_n)(u_n-u) \, dx\right| + \left|\int_\Omega (f(x,u)-f(x,u_n))u \,dx\right|\\
&\qquad \le  o(1) + \|f(\cdot,u)-f(\cdot,u_n)\|_{L^2(\Omega)}\|u\|_{L^2(\Omega)}=o(1).
\end{aligned}
\]
\end{proof}

\begin{remark}
If $\Omega$ has finite measure, the proof of \eqref{eq:fun-to-fu_Lp}, in its stronger form \eqref{eq:strongerth2}, can be simplified as follows. 
Let $q\in[1,\infty)$. Since $u_n\to u$ a.e. in $\Omega$ for a renamed subsequence, by the continuity of $f$ in \ref{f1} we have $f(x,u_n)\to f(x,u)$ for a.e. $x\in\Omega$. Moreover, by the first part of the proof, we know that $(f(x,u_n))$ is bounded in $L^{q+1}(\Omega)$. 
Thus, by \cite[Ex. 4.16]{Brezis} which holds on domains with finite measures, $f(x,u_n)\to f(x,u)$ in $L^q(\Omega)$.
\end{remark}

\subsection{Critical point theory for nonsmooth functionals}
We recall now the definitions of critical point and Palais-Smale sequence for nonsmooth functionals in Szulkin's critical point theory \cite{S}. 

Let $(E,\|\cdot\|)$ be a real Banach space, $\Psi:E\to(-\infty,+\infty]$ a proper (i.e., $\mathrm{Dom}(\Psi):=\{v\in E\,:\,\Psi(v)<+\infty\}\not=\emptyset$), convex, lower semincontinuous function. 
Let $\mathcal G\in C^1(E;\mathbb R)$ and $I:E\to \mathbb R$ be defined by $I:=\Psi-\mathcal G$.

\begin{definition}\label{def:critical_point}
A point $u\in E$ is said to be a \it{critical point of $I$} if $u\in\mathrm{Dom}(\Psi)$ and
\[
\Psi(u)-\Psi(v)-\mathcal G'(u)[u-v]\le 0\quad\mbox{for all }v\in E.
\]
\end{definition}

\begin{definition}\label{def:PS_sulkin}
A sequence $(u_n)\subset E$ is said to be a {\it Palais-Smale sequence at level $c$} ((PS)-sequence at level $c$ or (PS)$_c$-sequence for short) if 
\begin{equation}
I(u_n)\to c\in\R \mbox{ and }
\Psi(u_n)-\Psi(v)- \mathcal G'(u_n)[u_n-v]\leq \varepsilon_n \|u_n-v\| \ \ \forall v\in E,
\end{equation}
with $\varepsilon_n\to0$.

The functional $I$ {\it satisfies the Palais-Smale condition at level $c$} if any (PS)$_c$-sequence admits a convergent subsequence.
\end{definition}

Our main existence theorem relies on the following nonsmooth version of the mountain pass theorem without compactness which is obtained as an application of the general minimax principle \cite[Theorem 3.1]{LM} by Livrea and Marano, see also \cite{ABT, AM} for similar results.
\begin{theorem}
\label{thm:ArcoyaBereanuTorres}
Let $(E,\|\cdot\|)$ be a real Banach space, $I:E\to \mathbb R$ a functional defined by $I:=\Psi-\mathcal G$, where $\Psi:E\to(-\infty,+\infty]$ is a proper, convex function and $\mathcal G\in C^1(E;\mathbb R)$. If the following assumptions hold: 
\begin{itemize}
\item[(i)] $\mathrm{Dom}(\Psi)$ is closed and $\Psi\big|_{\mathrm{Dom}(\Psi)}$ is continuous;
\item[(ii)] $I(0)=0$ and there exists $\rho>0$ such that $I(u)>0$ for every $0< \|u\| \le \rho$;
\item[(iii)] there exists $e\in E$ with $\|e\|>\rho$ such that $I(e)\le 0$,
\end{itemize}
then there exists a (PS)-sequence $(u_n)\subset E$ at the mountain pass level 
\[ 
c_{mp}:=\inf_{\gamma\in \Gamma}\sup_{t\in[0,1]}I(\gamma(t)),
\]
where $\Gamma:=\{\gamma\in C([0,1];E)\,:\; \gamma(0)=0\neq \gamma(1), I(\gamma(1))\le 0\}$.
\end{theorem}
\begin{remark}
In the standard critical point theory for nonsmooth functionals, the function $\Psi$ is required to be lower semicontinuous. In the preceding theorem, this assumption is guaranteed by condition (i). 
\end{remark}
\begin{proof}[Proof of Theorem \ref{thm:ArcoyaBereanuTorres}]
We apply \cite[Theorem 3.1]{LM} with $\mathcal F=\{\gamma([0,1])\,:\, \gamma\in\Gamma\}$, $B=\{u\in E\,:\, I(u)\le 0\}$, and $F=\{u\in E\,:\, \|u\|=\rho\}$. Note that  $B$ is closed, as $I$ is lower semicontinuous by the above remark. Moreover, $\mathcal F$ is a homotopy-stable family of compact sets of $E$ with extended boundary $B$, namely for every $A\in\mathcal F$ and every $\eta\in C([0,1]\times E;E)$ such that $\eta(t,u)=u$ in $(\{0\}\times E)\cup ([0,1]\times B)$, one has $\eta(\{1\}\times A)\in \mathcal F$. To this aim it is enough to show that $\eta(1,\gamma(\cdot))\in\Gamma$, which is an easy consequence of the definitions of $\Gamma$ and $B$. The minimax level $c_{mp}$ introduced in the statement is finite because $\mathcal F\neq \emptyset$ by (iii) and $I(te)<+\infty$ for every $t\in [0,1]$, since $te\in \mathrm{Dom}(\Psi)$ by the convexity of $\Psi$. 
Furthermore, for every $\gamma\in\Gamma$, $(\gamma([0,1])\cap F )\setminus B \neq\emptyset$. Indeed, by (ii), $F\cap B=\emptyset$ and $\|\gamma(1)\|>\rho$; since $\gamma(0)=0$, by the continuity of $\gamma$, $(\gamma([0,1])\cap F) \setminus B =(\gamma([0,1])\cap F) \neq\emptyset$. Again by (ii), 
\[\sup_{u\in B}I(u)\le \inf_{u\in F}I(u).\]
Altogether, all the assumptions in \cite[Theorem 3.1]{LM} are satisfied and so, to every sequence $(\gamma_n)\subset\Gamma$ such that $\lim_{n\to\infty}\sup_{t\in[0,1]}I(\gamma_n(t))=c_{mp}$, corresponds a (PS)$_{c_{mp}}$-sequence. This concludes the proof.
\end{proof}

\subsection{Definition and regularity of the energy functional}
We will apply Theorem \ref{thm:ArcoyaBereanuTorres} to the real Banach space $(V;\|\cdot\|_V)$, see \eqref{eq:normV}, and the following functional $J_K:V\to (-\infty,+\infty]$,
\begin{equation}\label{eq:J_def2}
J_K(u):=\Psi_K(u)-\Phi(u)\quad \mbox{for every }u\in V,
\end{equation}
where $K$ is defined in the statement of Theorem \ref{thm:abstract} and $\Phi,\,\Psi:V\to \mathbb R$ are given respectively by 
\begin{equation}\label{eq:Phi-def}
\Phi(u):=\int_\Omega F(x,u) \,dx\quad\mbox{for all }u\in V,
\end{equation} 
and
\[
\Psi_K(u):=\begin{cases}
\frac12\|u\|_{H^1_\lambda(\Omega)}^2\ &\text{if }u\in K, \\
+\infty & \text{otherwise}.
\end{cases}
\]
We note that, within the cone $K$, the functional $J_K$ coincides with $J$, as defined in \eqref{eq:J-def}.

\begin{proposition}\label{prop:J-C1}
Under the assumptions of Theorem \ref{thm:abstract} the functional $\Phi$ in \eqref{eq:Phi-def} is well-defined and of class $C^1$.
\end{proposition}
\begin{proof} We first observe that $\Phi$ is well defined. Indeed, let $\alpha>0$ and $u\in V$. If $\Omega$ is unbounded,	by \eqref{eq:conseq-hyp-on-f} with $\alpha>0$, $\tilde\mu=1$ and $\nu=0$, the Cauchy-Schwarz inequality, and \eqref{eq:favourite_ineq},
\[
\begin{aligned}
\int_\Omega F(x,u)\, dx &\le \int_\Omega\left[Cu^2+|u|(e^{\alpha u^2}-1)\right]\, dx\\
&\le C\|u\|_{L^2(\Omega)}^2+\|u\|_{L^2(\Omega)}\left(\int_\Omega(e^{\alpha u^2}-1)^2\, dx\right)^{\frac{1}{2}} \\
&\leq C\|u\|_{L^2(\Omega)}^2+\|u\|_{L^2(\Omega)}\left(\int_\Omega(e^{2\alpha u^2}-1)\, dx\right)^{\frac{1}{2}}<\infty.
\end{aligned}
\]
Similarly, if $\Omega$ is bounded, one gets
\[
\int_\Omega F(x,u)\, dx 
\le C\left[\|u\|^2_{L^2(\Omega)}+\|u\|_{L^2(\Omega)}\left(\int_\Omega(e^{2\alpha u^2}-1)\, dx\right)^{\frac{1}{2}}\right]<\infty.
\]
Moreover, by standard calculations and using \eqref{eq:asterisco_mu_nu}, it is easy to prove that for every $u\in V$ the functional $\Phi$ admits the G\^ateaux derivative at $u$ defined by 
\[
\Phi'(u)[v]=\int_\Omega f(x,u)v \, dx\quad\mbox{for all }v\in V.
\]

It remains to prove the continuity of $\Phi'$. To this aim, let $u\in V$ and $(u_n)\subset V$ be such that $u_n\to u$ in $V$.  In particular, Lemma \ref{lem:convergenceOrlicz} applies to $(u_n)$ and $u$, thus
\begin{equation}
\label{eq:J-C1}
\lim_{n\to\infty} \|f(\cdot,u_n)-f(\cdot, u)\|_{L^2(\Omega)} =0.
\end{equation}
By the Cauchy-Schwarz inequality, \eqref{eq:L-embedded} and \eqref{eq:J-C1}, we have
\[
\begin{aligned}
\sup_{\substack{v\in V \\ \|v\|_{V}\leq1}}
\left| \int_\Omega \left( f(x,u_n)-f(x,u) \right) v \,dx \right|
&\leq 
\|f(\cdot,u_n)-f(\cdot, u)\|_{L^2(\Omega)} \sup_{\substack{v\in V \\ \|v\|_{V}\leq1}} \| v\|_{L^{2}(\Omega)} \\
&\le \|f(\cdot,u_n)-f(\cdot, u)\|_{L^2(\Omega)} \sup_{\substack{v\in V \\ \|v\|_{\mathcal{L}(\Omega)}\leq1}} \| v\|_{L^{2}(\Omega)} 
 \\
&\leq C\|f(\cdot,u_n)-f(\cdot, u)\|_{L^2(\Omega)} \to 0
\end{aligned}
\]
as $n\to\infty$, with $C$ as in Lemma \ref{lem:L-embedded} with $q=2$,  and we conclude that  $\Phi\in C^1(V;\R)$.
\end{proof}

\begin{remark}\label{rmk:JC1} 
As a consequence of Proposition \ref{prop:J-C1}, also the functional $J$ defined in \eqref{eq:J-def} is of class $C^1$. 
\end{remark}

\subsection{Proof of Theorem \ref{thm:abstract}}
For the sake of clarity, we split the proof into several steps.
First, we verify that all assumptions of Theorem \ref{thm:ArcoyaBereanuTorres} hold for the functional $J_K$, ensuring the existence of a (PS)-sequence at the level $c_{mp}$. Note that $c_{mp}$ coincides with the mountain pass level in the statement of the present theorem since $J_K=J$ in $K$ and $J_K=+\infty$ outside $K$, see the  definition of $c$ in \eqref{eq:mp-level} below. 

Next, we show that every (PS)-sequence has a strongly converging subsequence in the $H^1(\Omega)$-norm. We remark that our choice of applying Theorem \ref{thm:ArcoyaBereanuTorres}, instead of the standard Szulkin's nonsmooth mountain pass theorem, is motivated by the fact that Theorem \ref{thm:ArcoyaBereanuTorres} does not require compactness assumptions. 
Indeed, it is unclear whether $J_K$ satisfies the Palais-Smale condition in $V$,  as the main obstruction lies in the fact that $K$ is not a linear space, meaning that, in general, differences of elements in $K$ do not belong to $K$.

Thereafter, we prove that the limit function of the (PS)$_c$-sequence is a nontrivial critical point of $J_K$ in the sense of Szulkin's theory. Here, the convergence results from Lemma \ref{lem:convergences} play a crucial role. 

Finally, we show that the critical point obtained before is a weak solution of problem \eqref{eq:main}. This final step relies on the pointwise invariance property (ii) of the present theorem.

\noindent\textbf{Mountain pass geometry and existence of a (PS)$_c$-sequence.} 
As already mentioned, we will apply Theorem \ref{thm:ArcoyaBereanuTorres} with $E=V$ and $I=J_K=\Psi_K-\Phi$ as in \eqref{eq:J_def2}.
The functional $\Psi_K$ is convex, $\Psi_K\big|_{\mathrm{Dom}(\Psi_K)}$ is continuous, and the set $\mathrm{Dom}(\Psi_K)=K$ is closed. Moreover, by Proposition \ref{prop:J-C1}, $\Phi\in C^1(V;\R)$. Hence, $J_K$ satisfies the structural assumptions of Theorem \ref{thm:ArcoyaBereanuTorres} as well as condition (i) therein. 

We now prove that $J_K$ satisfies the mountain pass geometry assumptions (ii) and (iii) of Theorem \ref{thm:ArcoyaBereanuTorres}. As for (ii), clearly $J_K(0)=0$; for the remaining part we distinguish the proof depending on whether $\Omega$ is bounded or unbounded. 

{\it $\Omega$ bounded}: Since $J_K(u)=+\infty$ for every $u\in V\setminus K$, we can assume $u\in K$. Moreover, 
without loss of generality, we suppose that $\|u\|_{V}\le 1$. Now let $\alpha\in (0,1/2)$, $\nu>1$, and $\varepsilon\in (0,1)$. 
By \eqref{eq:conseq-hyp-on-f}, the variational characterization of $\lambda_1$, and  the Cauchy-Schwarz inequality
we have
\[
\begin{aligned}
J_K&(u) 
\ge  \frac{1}{2}\int_\Omega \left(|\nabla u|^2+\lambda u^2-(\lambda_1+\lambda)(1-\varepsilon)u^2-C|u|^{\nu+1}(e^{\alpha u^2}-1)\right)\, dx\\
&\ge \frac{1}{2}\|u\|^2_{H^1_\lambda(\Omega)}-\frac{1-\varepsilon}{2} \int_\Omega |\nabla u|^2 dx - \lambda \frac{1-\varepsilon}{2} \int_\Omega u^2 dx - C\int_\Omega |u|^{\nu+1}(e^{\alpha u^2}-1)\, dx\\
&\ge \frac{\varepsilon}{2}\|u\|^2_{H^1_\lambda(\Omega)}-C\|u\|_{L^{2(\nu+1)}(\Omega)}^{\nu+1}\left(\int_\Omega (e^{\alpha u^2}-1)^2\, dx\right)^{1/2}.
\end{aligned}
\]
Since, by Lemma \ref{lem:L-embedded}  and the definition of the $V$-norm,  we have $\|u\|_{L^{2(\nu+1)}(\Omega)}\le C \|u\|_{\mathcal L(\Omega)}\le C\|u\|_V$, 
we can refine the estimate for $J_K(u)$ above to obtain
\[
J_K(u) \ge \frac{\varepsilon}{2}\|u\|^2_{H^1_\lambda(\Omega)}-C\|u\|_{V}^{\nu+1}\left(\int_\Omega (e^{2\alpha u^2}-1)\, dx\right)^{1/2},
\]
where we have used inequality \eqref{eq:favourite_ineq} for the last term.
Moreover, since $\alpha<1/2$ and $\|u\|_{\mathcal L(\Omega)}\le \|u\|_V \le 1$,
\begin{equation}\label{eq:sup_exp}
\int_\Omega (e^{2\alpha u^2}-1)\, dx\le \int_\Omega (e^{u^2}-1)\, dx \le \int_\Omega \left[e^{\left(u/\|u\|_{\mathcal L(\Omega)}\right)^2}-1\right]\, dx \le 1,
\end{equation}
where we have used the definition of the Luxemburg norm $\|\cdot\|_{\mathcal L(\Omega)}$.
Hence, by combining the previous inequalities, we have  
\[
J_K(u)\ge \frac{\varepsilon}{2}\|u\|^2_{H^1_\lambda(\Omega)}-C\|u\|_V^{\nu+1}.
\]
Finally, using the equivalence of the norms $\|\cdot\|_V$ and $\|\cdot\|_{H^1_\lambda(\Omega)}$ in $K$ implied by assumption (i) of the present theorem (see Remark \ref{rem:equivalent-norms}), 
and the fact that $\nu+1> 2$, we easily conclude that condition (ii) of Theorem \ref{thm:ArcoyaBereanuTorres} is satisfied for $\rho$ sufficiently small.  

{\it $\Omega$ unbounded}: In this case the argument is slightly simplified since the corresponding version of \eqref{eq:conseq-hyp-on-f} is stronger than in the previous case. 
More precisely, applying \eqref{eq:conseq-hyp-on-f} with $\tilde{\mu}=\mu$, $\alpha\in(0,1/2)$, and $\nu>1$, and proceeding as above, we obtain for every $u\in K$
\[
\begin{aligned}
J_K(u) 
&\ge  \frac{1}{2}\|u\|^2_{H^1_\lambda(\Omega)} - C\|u\|^{\mu+1}_{L^{\mu+1}(\Omega)} - \|u\|^{\nu+1}_{L^{2(\nu+1)}(\Omega)}\left(\int_\Omega(e^{2\alpha u^2}-1) \, dx\right)^{1/2}\\
&\ge \frac{1}{2}\|u\|^2_{H^1_\lambda(\Omega)} - C\left(\|u\|^{\mu+1}_V + \|u\|^{\nu+1}_V\right),
\end{aligned}
\]
where we have used \eqref{eq:sup_exp} that holds also in the unbounded case with the same proof. 
Since $\mu+1> 2$ and $\nu+1> 2$, and using again the equivalence of the norms in $K$, for $\rho$ sufficiently small condition (ii) of Theorem \ref{thm:ArcoyaBereanuTorres} is satisfied.  
\smallskip

As for (iii), let $\bar u\in K\setminus\{0\}$. By \eqref{f4new}, for every $t>0$ we have
\[
\begin{aligned}
J_K(t\bar u) 
\le \frac{t^2}{2}\|\bar u\|_{H^1_\lambda(\Omega)}^2-t^{\sigma}\int_\Omega c_1(x)|\bar u|^{\sigma}dx +\|c_2\|_{L^1(\Omega)}.
\end{aligned}
\]
Since $\sigma>2$, for $t$ large enough condition (iii) is verified with $e=t\bar{u}$.  	

In conclusion, all the assumptions of Theorem \ref{thm:ArcoyaBereanuTorres} are satisfied and so there exists a (PS)-sequence at level 
\begin{equation}\label{eq:mp-level}
c:=\inf_{\gamma\in \Gamma}\sup_{t\in[0,1]}J_K(\gamma(t))=\inf_{\gamma\in \Upsilon}\sup_{t\in [0,1]}J_K(\gamma(t))=\inf_{\gamma\in\Upsilon}\sup_{t\in[0,1]}J(\gamma(t)).
\end{equation}
Note that the first equality in \eqref{eq:mp-level} holds since $J_K(u)=+\infty$ for $u\not\in K$, even if $\Upsilon \subsetneq \Gamma=\{\gamma\in C([0,1];V)\,:\,\gamma(0)=0\neq\gamma(1),\, J_K(\gamma(1))\le 0\}$, while the last equality derives from the fact that $J_K=J$ in $K$. 
\smallskip 

\noindent\textbf{Convergence of (PS)$_c$-sequences.} 
Let us now show that every (PS)$_c$-sequence is bounded. For every $u\in K$ and $v\in V$, let
\[
G(u,v):=\Psi_K(u)-\Psi_K(v)- \Phi'(u)[u-v],
\]
so that if $v\in K$
\begin{equation}\label{eq:def_G}
G(u,v)=\frac{1}{2}\|u\|^2_{H^1_\lambda(\Omega)}-\frac{1}{2}\|v\|^2_{H^1_\lambda(\Omega)} - \int_\Omega f(x,u)(u-v)\, dx.
\end{equation}
Let $\sigma$ be as in \ref{f3} and
\begin{equation}\label{eq:gamma_delta_def}
\kappa\in (1,\sigma-1) \quad\text{ and }\quad \delta\in \left(
\frac{1}{\sigma(\kappa-1)},\frac{1}{\kappa^2-1}\right).
\end{equation}
Let $(u_n)\subset K$ be a Palais-Smale sequence at mountain pass level $c$, i.e. 
\begin{equation}\label{eq:def_PS_sulkin}
J_K(u_n)\to c\in\R, \qquad
G(u_n,v)\leq \varepsilon_n \|u_n-v\|_V \ \mbox{ for all }v\in V,
\end{equation}
with $\varepsilon_n\to0$. 
Then, taking $v= \kappa u_n$ in \eqref{eq:def_PS_sulkin} and using that $v\in K$ since $K$ is a cone, we have
\begin{equation}\label{eq:PS_calc}
J_K(u_n)+\delta G(u_n,\kappa u_n) \leq c+1 +\delta(\kappa-1) \varepsilon_n \|u_n\|_V\quad\mbox{for large $n$.}
\end{equation}
In order to have an estimate from below of the quantity $J_K(u_n)+\delta G(u_n,\kappa u_n)$, we distinguish the two cases depending on whether $\Omega$ is bounded or unbounded. If $\Omega$ is unbounded, by \eqref{eq:J_def2}, \eqref{eq:def_G}, \ref{f3}-(u), \ref{f1}, and \eqref{eq:gamma_delta_def},
\[
\begin{aligned}
J_K(u_n)+\delta G(u_n,\kappa u_n) &=
\frac12\|u_n\|^2_{H^1_\lambda(\Omega)}- \int_\Omega F(x,u_n) \, dx 
+\frac{\delta }{2} \|u_n\|^2_{H^1_\lambda(\Omega)} \\
&\phantom{==}- \frac{\delta \kappa^2}{2}\|u_n\|^2_{H^1_\lambda(\Omega)}
+\delta(\kappa-1)\int_\Omega f(x,u_n)u_n \,dx \\
&\geq \frac{1+\delta -\delta\kappa^2}{2} \|u_n\|^2_{H^1_\lambda(\Omega)}+\left[\delta(\kappa-1)-\frac{1}{\sigma}\right]\int_\Omega f(x,u_n)u_n \,dx\\
&\geq \frac{1+\delta -\delta\kappa^2}{2} \|u_n\|^2_{H^1_\lambda(\Omega)}.
\end{aligned}
\]
While, if $\Omega$ is bounded, reasoning as above, using \ref{f3}-(b) instead of \ref{f3}-(u), we get
\[
\begin{aligned}
J_K(u_n)+\delta G(u_n,\kappa u_n) &=
\frac12\|u_n\|^2_{H^1_\lambda(\Omega)} 
+\frac{\delta }{2} \|u_n\|^2_{H^1_\lambda(\Omega)} - \frac{\delta \kappa^2}{2}\|u_n\|^2_{H^1_\lambda(\Omega)}\\
&\phantom{==}+\int_{\left\{|u_n|<\overline{M}\right\}} \left[\delta(\kappa-1)f(x,u_n)u_n - F(x,u_n)\right] \, dx \\
&\phantom{==}+\int_{\left\{|u_n|\ge\overline{M}\right\}} \left[\delta(\kappa-1)f(x,u_n)u_n - F(x,u_n)\right] \, dx\\
&\geq \frac{1+\delta -\delta\kappa^2}{2} \|u_n\|^2_{H^1_\lambda(\Omega)}-C\\
&\phantom{==}+\left[\delta(\kappa-1)-\frac{1}{\sigma}\right]\int_{\{|u_n|\ge\overline{M}\}} f(x,u_n)u_n \,dx\\
&\geq \frac{1+\delta-\delta\kappa^2}{2} \|u_n\|^2_{H^1_\lambda(\Omega)} -C.
\end{aligned}
\]
In both cases, as $1+\delta -\delta\kappa^2>0$ by \eqref{eq:gamma_delta_def}, by combining the estimate from above \eqref{eq:PS_calc} with the estimate from below, and  using the equivalence of the norms $\|\cdot\|_V$ and $\|\cdot\|_{H^1_\lambda(\Omega)}$ in $K$ (see Remark \ref{rem:equivalent-norms}), we infer that $(u_n)$ is bounded in $H^1(\Omega)$. In particular, there exists $u$ such that $u_n\rightharpoonup u$ in $H^1_0(\Omega)$ up to a subsequence (which is still denoted by $(u_n)$ below). Finally, since $K$ is weakly closed, $u\in K$.

Now we prove that the convergence is strong in $H^1_0(\Omega)$. 
We take $v=u$ in the inequality in \eqref{eq:def_PS_sulkin} to get, by \eqref{eq:fun_to_0} and the boundedness of $(u_n)$,
\[
\frac{1}{2}\left(\|u_n\|_{H^1_\lambda(\Omega)}-\|u\|_{H^1_\lambda(\Omega)}\right)\le \int_\Omega f(x,u_n)(u_n-u)\, dx+\varepsilon_n\|u_n-u\|_V\to 0\mbox{ as }n\to\infty.
\]
Together with the weak lower semicontinuity of the norm, this allows us to conclude that $\|u_n\|_{H^1_\lambda}\to \|u\|_{H^1_\lambda}$, which in turn implies that $u_n\to u$ in $H^1_0(\Omega)$ by the weak convergence.
We remark that, since in general $u_n-u\not \in K$, the norm equivalence in $K$ stated in Remark \ref{rem:equivalent-norms} is not helpful to conclude that the convergence also holds with respect to the $V$-norm.
\smallskip 

\noindent\textbf{The limit point $u$ is a critical point of $J_K$ at level $c$.}
We first prove that the limit function $u$ has energy $J_K(u)=c$. To this aim, it is enough to pass in the limit in the first condition in \eqref{eq:def_PS_sulkin}, namely 
\[
J_K(u_n)=\frac12\|u_n\|_{H^1_\lambda(\Omega)}^2-\int_\Omega F(x,u_n)\, dx\to J_K(u) \quad\mbox{as }n\to\infty.
\]
We have already proved that the term involving the norm converges to the corresponding term $\frac12\|u\|_{H^1_\lambda(\Omega)}^2$, it remains to consider the term involving $F$. By the continuity of $F(x,\cdot)$, we know that $F(x,u_n)\to F(x,u)$ a.e. in $\Omega$ for a renamed subsequence. By \ref{f1} and \ref{f3}, we can dominate each  $F(x,u_n)$ as follows:
\begin{equation*}
0\le F(x,u_n)\le \begin{cases} C+\frac1\sigma  f(x,u_n) u_n\quad&\mbox{if $\Omega$ is bounded}\\
\frac1\sigma f(x,u_n) u_n\quad&\mbox{if $\Omega$ is unbounded.}
\end{cases}
\end{equation*}
Hence, combining the domination with the convergence in \eqref{eq:convergence_fun-un}, by the General Dominated Convergence Theorem we get
\[
\int_\Omega F(x,u_n)\, dx\to \int_\Omega F(x,u)\, dx \quad\mbox{as }n\to\infty.
\]
This concludes the proof of the convergence $J_K(u_n)\to J_K(u)$ thus giving $J_K(u)=c$.  

We finally prove that the limit $u$ is a critical point for $J_K$ in the sense of Definition \ref{def:critical_point}. We observe that by the definition of $\Psi_K$, if $v\in V\setminus K$, the inequality involved in Definition \ref{def:critical_point} is trivially verified. We consider $v\in K$. In this case, we will pass to the limit in the inequality part of \eqref{eq:def_PS_sulkin}. Indeed, by \eqref{eq:def_PS_sulkin} we have
\begin{equation}
\label{eq:da-passare-al-lim_derivate}
\frac{1}{2}\|u_n\|^2_{H^1_\lambda(\Omega)}-\frac{1}{2}\|v\|^2_{H^1_\lambda(\Omega)}-\int_\Omega f(x,u_n)(u_n-v)\, dx\le \varepsilon_n \|u_n-v\|_V.
\end{equation}
Letting $n\to\infty$, we easily see that the right-hand side goes to 0, by the boundedness of $(u_n)$ in $V$. In view of the convergence $\|u_n\|_{H^1_\lambda}\to \|u\|_{H^1_\lambda}$, it remains to consider the term involving the nonlinearity:
\[
\begin{aligned}
&\left|\int_\Omega f(x,u_n)(u_n-v)\, dx-\int_\Omega f(x,u)(u-v)\, dx\right|\\
&\qquad\le \left|\int_\Omega f(x,u_n)(u_n-v)\, dx-\int_\Omega f(x,u_n)(u-v)\, dx\right|\\
&\qquad\qquad +\left|\int_\Omega f(x,u_n)(u-v)\, dx-\int_\Omega f(x,u)(u-v)\, dx\right|\\
&\qquad
\le \left|\int_\Omega f(x,u_n)(u_n-u)\, dx\right| + \|f(x,u_n)-f(x,u)\|_{L^2(\Omega)}\|u-v\|_{L^2(\Omega)}.
\end{aligned}
\] 
By \eqref{eq:fun_to_0} and \eqref{eq:fun-to-fu_Lp} we have that the right-hand side converges to zero, and so, by \eqref{eq:da-passare-al-lim_derivate}, we conclude that $u$ is a critical point of $J_K$ in the sense of Definition \ref{def:critical_point}.
\smallskip 

\noindent\textbf{The critical point $u$ is a solution.} To show that $u$ solves \eqref{eq:main}, we exploit the pointwise invariance as in \cite[proof of Theorem 2.1]{CowanMoameni2022}; we report the proof here for the sake of completeness. 
On the one hand, assumption (ii) of the present theorem provides the existence of $v\in K$ such that
\begin{equation}\label{eq:linear-weak}
\int_\Omega (\nabla v\cdot \nabla \eta +\lambda v\eta ) \, dx
=\int_\Omega f(x,u)\eta \,dx
\quad\text{ for every } \eta \in V.
\end{equation}
Choosing $\eta=u-v$ in the previous equality, we obtain
\begin{equation}\label{eq:conseq-pointwise-invariance}
\int_\Omega (\nabla u\cdot \nabla v +\lambda uv ) \, dx
-\|v\|^2_{H^1_\lambda(\Omega)}
=\int_\Omega f(x,u)(u-v) \,dx.
\end{equation}
On the other hand,  being $u$ a critical point of $J_K$, by Definition \ref{def:critical_point} it holds
\[
\frac12 \|u\|^2_{H^1_\lambda(\Omega)} - \frac12 \|v\|^2_{H^1_\lambda(\Omega)} - \int_\Omega f(x,u)(u-v) \,dx\le 0.
\]
By combining the last inequality with \eqref{eq:conseq-pointwise-invariance} we obtain
\[
\|u-v\|^2_{H^1_\lambda(\Omega)} \leq 0,
\]
thus implying that $u=v$ a.e. in $\Omega$.  This, together with \eqref{eq:linear-weak}, implies that $u$ solves \eqref{eq:main}. In particular, $u$ is a critical point of the functional $J$ defined in \eqref{eq:J-def}.

\section{Application}\label{sec:proof-applied-thm}
In this section, we will prove Theorem \ref{thm:applied1} as an application of Theorem \ref{thm:abstract}.
Hereafter, let $\Omega=A$ be as in \eqref{eq:A-def}, with $0<R_0 <R_1\le \infty$ and $N\geq3$, let $\mathcal K$ be as in \eqref{eq:K-def}, and assume that the hypotheses stated in Theorem \ref{thm:applied1} are satisfied.

We first observe that, proceeding exactly as in \cite[Lemma 2.3]{BCNW2024}, one can prove the following. 
\begin{lemma}\label{le:Kcone}
The convex cone $\mathcal K$ is closed with respect to the $H^1(A)$-norm; as a consequence, it is weakly closed. 
\end{lemma}

Let $r,\,\theta$ be the polar coordinates introduced in \eqref{eq:def_r-theta}. We note that the functions $u:A\to \mathbb R$ belonging to $\mathcal K$ 
can also be written in terms of the cartesian coordinates in $\mathbb R^2$
\begin{equation}\label{eq:s-t}
\xi:=r\cos\theta=\sqrt{x_1^2+\ldots+x_{N-1}^2}, \quad
\zeta:=r\sin\theta=|x_N|, 
\end{equation}
with $(\xi,\zeta) \in Q$, where
\begin{equation}\label{eq:defQ}
Q:= \{(\xi,\zeta)\in\R^2:\, R_0^2<\xi^2+\zeta^2<R_1^2, \, \xi\geq0, \, \zeta\geq0\}. 
\end{equation}
We will use the notation
\[u(x)=\mathfrak{u}(r,\theta)=\tilde{\mathfrak{u}}(\xi,\zeta).\]

\subsection{Some useful change of variable formulas}\label{subsec:chofvar}

We recall below some useful formulas which will be often used in what follows.

Changing variables in the integral, we get 
\begin{equation}\label{eq:changeintegral}
\begin{aligned}
\int_A u(x)\, dx & = 2\omega_{N-2} \int_Q \tilde{\mathfrak{u}}(\xi,\zeta) \xi^{N-2}\,d\xi \,d\zeta \\
& = 2\omega_{N-2} \int_{R_0}^{R_1} \int_0^{\frac{\pi}{2}} \mathfrak{u}(r,\theta) (\cos\theta)^{N-2} r^{N-1} \,dr\,d\theta,
\end{aligned}
\end{equation}
where $\omega_{N-2}$ denotes the surface measure of the sphere $\mathbb{S}^{N-2} \subseteq \R^{N-1}$.

Now, since 
\begin{equation}\label{eq:gradient-s-t}
u_{x_i}=
\begin{cases}
\displaystyle \frac{x_i}{\xi} \tilde{\mathfrak{u}}_\xi & \qquad i=1,\ldots,N-1 \vspace{0.2cm}\\
\displaystyle \frac{x_N}{\zeta} \tilde{\mathfrak{u}}_\zeta & \qquad i=N,
\end{cases}
\end{equation}
passing to polar coordinates, we have
\[
\tilde{\mathfrak{u}}_\xi=\cos\theta\, \mathfrak{u}_r - \frac{\sin\theta}{r}\mathfrak{u}_\theta, \qquad 
\tilde{\mathfrak{u}}_\zeta=\sin\theta\, \mathfrak{u}_r + \frac{\cos\theta}{r}\mathfrak{u}_\theta.
\]
Therefore,
\begin{equation}\label{eq:gradient-r-theta}
\vert \nabla u \vert^2 = \tilde{\mathfrak{u}}_\xi^2 + \tilde{\mathfrak{u}}_\zeta^2 = \mathfrak{u}_r^2 + \frac{\mathfrak{u}_\theta^2}{r^2}.
\end{equation}
Furthermore, we have 
\[
u_{x_ix_i}=
\begin{cases}
\displaystyle \frac{x_i^2}{\xi^2}\tilde{\mathfrak{u}}_{\xi\xi}+\frac{1}{\xi}\left(1-\frac{x_i^2}{\xi^2}\right)\tilde{\mathfrak{u}}_\xi & \qquad i=1,\ldots,N-1 \vspace{0.2cm}\\
\displaystyle \tilde{\mathfrak{u}}_{\zeta\zeta} & \qquad i=N,
\end{cases}
\]
and so, expressing the Laplacian in polar coordinates, namely
\[
\tilde{\mathfrak{u}}_{\xi\xi}+\tilde{\mathfrak{u}}_{\zeta\zeta}=\mathfrak{u}_{rr}+\frac{1}{r}\mathfrak{u}_r+\frac{1}{r^2}\mathfrak{u}_{\theta\theta},
\]
we finally get
\begin{equation}\label{eq:Deltau_s-r-th}
\begin{aligned}
\Delta u =\tilde{\mathfrak{u}}_{\xi\xi}+\tilde{\mathfrak{u}}_{\zeta\zeta}+\frac{N-2}{\xi}\tilde{\mathfrak{u}}_\xi
 = \mathfrak{u}_{rr}+ \frac{N-1}{r}\mathfrak{u}_r-\frac{(N-2)\tan\theta}{r^2}\mathfrak{u}_\theta +\frac{1}{r^2} \mathfrak{u}_{\theta\theta}.
\end{aligned}
\end{equation}

\subsection{Embeddings for $\mathcal K$}\label{subsec:verification(i)}

In this subsection, we will prove the validity of assumption (i) of Theorem \ref{thm:abstract}. Our aim is to investigate the continuous embedding 
of the cone $\mathcal K$ in the Orlicz space $\mathcal M(A)$. The idea of exploring this type of embedding is motivated by the fact that the cone $\mathcal K$ is continuously embedded in  any Lebesgue space $L^q(A)$ with $q\in [2,\infty)$, as shown in  \cite[Lemma 2.3]{BCNW2023} for the bounded case $R_1<\infty$ and in \cite[Lemma 2.4]{BCNW2024} for $R_1=\infty$.  Moreover, we do not expect that $\mathcal K$ is embedded in $L^\infty(A)$.
This suggests that there should be a smaller target Orlicz space, rather than Lebesgue spaces, for the embedding of the cone $\mathcal{K}$. 
Indeed, we will prove that, thanks to the symmetry and monotonicity properties of functions belonging to $\mathcal K$, a Pohozaev-Trudinger-Moser inequality holds in the cone and this will give the continuous embedding of $\mathcal K$ in $\mathcal M(A)$ required in (i) of Theorem \ref{thm:abstract}. This crucial inequality is deduced by the Trudinger-Moser inequality in $\mathbb R^2$ established by Ruf \cite{Ruf}, that we report below for the reader's convenience, together with a useful consequence. 

\begin{theorem}[{\cite[Theorem 1.1]{Ruf}}]\label{thm:Ruf}
It holds  
\begin{equation}\label{eq:Ruf}
  \sup_{\substack{u\in H^1(\mathbb R^2)\\\|u\|_{H^1(\mathbb R^2)}\le 1}}\int_{\mathbb R^2} (e^{4\pi u^2}-1)\, dx =: C^*  
\end{equation} 
and the exponent $4\pi$ is sharp. 
As a consequence, if $u\in H^1(\mathbb R^2)$ then 

\begin{equation}\label{eq:int-finito}
\int_{\mathbb R^2} (e^{\alpha u^2}-1) \, dx<\infty\quad\mbox{for every $\alpha>0$}.
\end{equation}
\end{theorem}

\begin{proof} The first part of the statement was established in \cite{Ruf}. Here, we focus specifically on deriving the consequence stated in \eqref{eq:int-finito}.
Let $u\in H^1(\mathbb R^2)$, $\alpha>0$, and $\varepsilon\in(0,\sqrt{\pi/\alpha})$ be fixed. By density, there exists $\varphi_\varepsilon\in C^\infty_0(\mathbb R^2)$ such that $\|u-\varphi_\varepsilon\|_{H^1(\mathbb R^2)}\le\varepsilon$. Therefore, since $\alpha u^2 = \alpha(u - \varphi_\varepsilon + \varphi_\varepsilon)^2\le 2\alpha(u - \varphi_\varepsilon)^2+2\alpha\varphi_\varepsilon^2$, by Young's inequality, and by \eqref{eq:Ruf} we get
\[
\begin{aligned}
\int_{\mathbb R^2} (e^{\alpha u^2}-1) \, dx &\le \int_{\mathbb R^2}(e^{2\alpha (u - \varphi_\varepsilon)^2}e^{2\alpha \varphi_\varepsilon^2}-1) \, dx \\
&\le \frac{1}{2}\int_{\mathbb R^2} (e^{4\alpha(u - \varphi_\varepsilon)^2}-1)\, dx + \frac{1}{2}\int_{\mathbb R^2} (e^{4\alpha \varphi_\varepsilon^2}-1)\, dx\\
&\le \frac{1}{2}\int_{\mathbb R^2} \left(e^{4\alpha\varepsilon^2\left(\frac{u - \varphi_\varepsilon}{\|u - \varphi_\varepsilon\|_{H^1(\mathbb R^2)}}\right)^2}-1\right)\, dx + 2\alpha\int_{\mathbb R^2} \varphi_\varepsilon^2 e^{4\alpha\varphi_\varepsilon^2}\, dx\\
&\le \frac{1}{2}C^*+ 2\alpha e^{4\alpha\|\varphi_\varepsilon\|^2_{L^\infty(\Omega)}}\|\varphi_\varepsilon\|^2_{L^2(\Omega)}<\infty,
\end{aligned}
\]
where in the third step we have used \eqref{eq:second-favourite-ineq}.
This concludes the proof.
\end{proof}

We are now ready to prove the validity of a Pohozaev-Trudinger-Moser type inequality in $\mathcal K$. 

\begin{proposition}\label{cor:Trud-Moser1}
There exist $\alpha^*>0$ and $C^*_\mathcal K:=C^*_\mathcal K(R_0,N)>0$ such that for all $\alpha\le \alpha^*$
\begin{equation}\label{eq:Trud-Moser-type}
\sup_{\substack{u\in\mathcal K\\ \|u\|_{H^1(A)}\le 1}}\int_A (e^{\alpha u^2}-1)\, dx \le C^*_{\mathcal K}.
\end{equation}
As a consequence, if $u\in \mathcal K$ then

\begin{equation}\label{eq:int-finito-K}
\int_A (e^{\alpha u^2}-1) \, dx<\infty\quad\mbox{for every $\alpha>0$}.
\end{equation} 
\end{proposition}
\begin{proof} 
Let $Q$ be as in \eqref{eq:defQ} and let
\[
\begin{aligned}
Q_{\frac{\pi}{3}}:=&\left\{(\xi,\zeta)\in Q\,:\, \arctan\frac{\zeta}{\xi}\in\left(0,\frac{\pi}{3}\right)\right\}\\&\quad=\left\{(r,\theta)\in \mathbb R^2\,:\, R_0<r<R_1,\,\theta\in\left(0,\frac{\pi}{3}\right)\right\}.
\end{aligned}
\]
Notice that 
\begin{equation}
\label{eq:s>}
\xi>\frac{R_0}{2}\quad \mbox{ for all }(\xi,\zeta)\in Q_{\frac{\pi}{3}}. 
\end{equation}
Now, let $\alpha>0$ and $u\in\mathcal K$, with $\|u\|_{H^1(A)}\le 1$, then using \eqref{eq:changeintegral} we get 
\begin{equation}\label{eq:estimating-int}
\begin{aligned}
\int_A (e^{\alpha u^2}-1)\, dx &= 2\omega_{N-2}\int_{R_0}^{R_1}\int_0^{\frac{\pi}{3}}(e^{\alpha\mathfrak u^2}-1)(\cos\theta)^{N-2} r^{N-1}dr \,d\theta\\
& \quad + 2\omega_{N-2}\int_{R_0}^{R_1}\int_{\frac{\pi}{3}}^{\frac{\pi}{2}}(e^{\alpha\mathfrak u^2}-1)(\cos\theta)^{N-2} r^{N-1}dr \,d\theta\\
& =: I_{Q_{\frac{\pi}{3}}} + I_{Q_{\frac{\pi}{3}}^c}.
\end{aligned}
\end{equation}
We observe that, by the angular monotonicity of $u\in\mathcal{K}$,
\[
\begin{aligned}
I_{Q^c_{\frac{\pi}{3}}} &\le 2\omega_{N-2}\int_{R_0}^{R_1}\int_{\frac{\pi}{3}}^{\frac{\pi}{2}}\left(e^{\alpha \left(\mathfrak u(r,\theta-\frac{\pi}{4})\right)^2 }-1\right)\left(\cos\left(\theta-\frac{\pi}{4}\right)\right)^{N-2} r^{N-1}dr\,d\theta
\\
&= 2\omega_{N-2}\int_{R_0}^{R_1}\int_{\frac{\pi}{12}}^{\frac{\pi}{4}}\left(e^{\alpha(\mathfrak u(r,\theta))^2}-1\right)(\cos\theta)^{N-2} r^{N-1}dr\le I_{Q_{\frac{\pi}{3}}}.
\end{aligned}
\]
Hence, by \eqref{eq:estimating-int} and using again \eqref{eq:changeintegral}, we have 
\begin{equation*}
\begin{aligned}
\int_A (e^{\alpha u^2}-1)\, dx\le 2I_{ Q_{\frac{\pi}{3}}}=4\omega_{N-2}\int_{Q_\frac{\pi}{3}}(e^{\alpha\tilde{\mathfrak u}^2}-1)\xi^{N-2}\,d\xi \,d\zeta.
\end{aligned}
\end{equation*}
We introduce now $\mathfrak{\tilde{v}}(\xi,\zeta):=\xi^{m} \tilde{\mathfrak{u}}(\xi,\zeta)$, with $m:=(N-2)/2$. In view of \eqref{eq:s>}, we can estimate in $Q_{\frac{\pi}{3}}$
\begin{equation}
\label{eq:est-exp-s}
\begin{aligned}
(e^{\alpha \tilde{\mathfrak{u}}^2}-1)\xi^{N-2}&=(e^{\alpha \xi^{-2m}\mathfrak{\tilde{v}}^2}-1)\xi^{N-2}=\xi^{N-2}\sum_{i=1}^\infty\frac{(\alpha \mathfrak{\tilde{v}}^2)^i}{i!}\xi^{-(N-2)i}\\
&\le \left(\frac{R_0}{2}\right)^{N-2}\sum_{i=1}^\infty\frac{(\alpha \mathfrak{\tilde{v}}^2)^i}{i!}\left(\frac{R_0}{2}\right)^{-(N-2)i}\\
&= \left(\frac{R_0}{2}\right)^{N-2}\left(e^{\bar\alpha\mathfrak{\tilde{v}}^2}-1\right),
\end{aligned}
\end{equation}
where $\bar\alpha:=\alpha(2/R_0)^{N-2}$, and so 
\begin{equation}
\label{eq:u-v}
\int_A (e^{\alpha u^2}-1)\, dx\le 4\omega_{N-2}\left(\frac{R_0}{2}\right)^{N-2}\int_{Q_{\frac{\pi}{3}}}(e^{\bar\alpha \mathfrak{\tilde{v}}^2}-1)\,d\xi\, d\zeta.
\end{equation}
In order to provide a uniform estimate of the right hand side of the previous inequality, we will use the Pohozaev-Trudinger-Moser inequality \eqref{eq:Ruf} in $\mathbb R^2$. To this aim, we need both an estimate on the $H^1$-norm of $\mathfrak{\tilde{v}}$, and to extend $\mathfrak{\tilde{v}}$ on the whole $\R^2$. 
Concerning the $H^1$-norm of $\mathfrak{\tilde{v}}$, since 
\[
|\nabla \mathfrak{\tilde{v}}|^2=\xi^{2(m-1)}\left[(m\tilde{\mathfrak u}+\xi\tilde{\mathfrak u}_\xi)^2+\xi^2\tilde{\mathfrak u}_\zeta^2\right]\le 2(\xi^{2m}|\nabla \tilde{\mathfrak{u}}|^2+m^2 \xi^{2(m-1)}\tilde{\mathfrak{u}}^2),
\]
using the definition of $m$, \eqref{eq:s>},  and \eqref{eq:changeintegral}, we get
\begin{equation}\label{eq:gradient_v_est}
\begin{aligned}
\int_{Q_{\frac{\pi}{3}}}&(|\nabla \mathfrak{\tilde{v}}|^2+\mathfrak{\tilde{v}}^2)\,d\xi \,d\zeta\\
&\le 2\int_{Q_{\frac{\pi}{3}}} \left[\xi^{2m-(N-2)}|\nabla \tilde{\mathfrak{u}}|^2+\left(m^2 \xi^{2(m-1)-(N-2)}+\xi^{2m-(N-2)}\right)\tilde{\mathfrak{u}}^2\right] \xi^{N-2} \,d\xi \,d\zeta\\
&\le 2\int_{Q_{\frac{\pi}{3}}}\left[|\nabla \tilde{\mathfrak{u}}|^2+\left(\left(\frac{N-2}{R_0}\right)^2+1\right)\tilde{\mathfrak{u}}^2\right]\xi^{N-2}\,d\xi \,d\zeta\\
&\le \frac{1}{\omega_{N-2}}\left[\left(\frac{N-2}{R_0}\right)^2+1\right]\int_{A}(|\nabla u|^2+u^2)\, dx \le \bar k,
\end{aligned}
\end{equation}
where $\bar k:=\frac{1}{\omega_{N-2}}\left[\left(\frac{N-2}{R_0}\right)^2+1\right]$, as $\|u\|_{H^1(A)}\le 1$.
Now, by  \cite[Theorem 5.24]{AF}, there exists a continuous extension operator $T: H^1(Q_{\frac{\pi}{3}})\to H^1(\mathbb R^2)$, so that 
\begin{equation}\label{eq:extensionR2}
\|T\mathfrak{\tilde{v}}\|_{H^(\mathbb R^2)}\le \bar c\|\mathfrak{\tilde{v}}\|_{H^1(Q_{\frac{\pi}{3}})},
\end{equation}
with $\bar c$ only depending on the region $Q_{\frac{\pi}{3}}$.
Therefore,  by \eqref{eq:gradient_v_est},
\[
\begin{aligned}
\int_{Q_{\frac{\pi}{3}}}(e^{\bar\alpha\mathfrak{\tilde{v}}^2}-1)\,d\xi \,d\zeta &
\le \int_{Q_{\frac{\pi}{3}}}\left[
\mathrm{exp}\left(\frac{\bar\alpha\bar k\mathfrak{\tilde{v}}^2}{\|\mathfrak{\tilde{v}}\|^2_{H^1(Q_{\frac{\pi}{3}})}}\right)-1\right]\,d\xi \,d\zeta\\
&\le \int_{\mathbb R^2}\left[
\mathrm{exp}\left(\frac{\bar\alpha\bar k\bar c^2 (T\mathfrak{\tilde{v}})^2}{\|T\mathfrak{\tilde{v}}\|^2_{H^1(\mathbb R^2)}}\right)-1\right]\,d\xi \,d\zeta,
\end{aligned}
\]
with $\bar c$ as in \eqref{eq:extensionR2}.
Now, if we choose $\alpha>0$ so small that
\[\bar\alpha\bar k\bar c^2=\alpha\left(\frac{2}{R_0}\right)^{N-2}\bar k\bar c^2\le 4\pi,\]
then \eqref{eq:Ruf} provides
\[
\int_{Q_{\frac{\pi}{3}}}(e^{\bar\alpha\mathfrak{\tilde{v}}^2}-1)\,d\xi \,d\zeta \le C^*.
\]
Altogether, in view of \eqref{eq:u-v},  relation \eqref{eq:Trud-Moser-type} holds true with 
\[
C^*_\mathcal K = 4\omega_{N-2}C^*\left(\frac{R_0}{2}\right)^{N-2}.
\]

Finally, to prove \eqref{eq:int-finito-K} let $u\in \mathcal K$ and $\alpha>0$. We can argue as above to get
\begin{equation*}
\int_A (e^{\alpha u^2}-1)\, dx \le 4\omega_{N-2}\int_{Q_{\frac{\pi}{3}}}(e^{\alpha \tilde{\mathfrak{u}}^2}-1)\xi^{N-2} \,d\xi\,d\zeta
\end{equation*}
and then, considering again $\mathfrak{\tilde{v}}(\xi,\zeta)=\xi^{m}\tilde{\mathfrak{u}}(\xi,\zeta)$, we know that $\mathfrak{\tilde{v}}\in H^1(Q_{\frac{\pi}{3}})$ and that there exists an extension $T\mathfrak{\tilde{v}}\in H^1(\mathbb R^2)$. Consequently, reasoning as in \eqref{eq:u-v}, we obtain
\begin{equation*}
\begin{aligned}
\int_A (e^{\alpha u^2}-1)\, dx &\le 4\omega_{N-2}\left(\frac{R_0}{2}\right)^{N-2}\int_{Q_{\frac{\pi}{3}}}(e^{\bar\alpha \mathfrak{\tilde{v}}^2}-1) \,d\xi\,d\zeta\\
&\le 4\omega_{N-2}\left(\frac{R_0}{2}\right)^{N-2}\int_{\mathbb R^2}(e^{\bar\alpha (T\mathfrak{\tilde{v}})^2}-1) \,d\xi\,d\zeta,
\end{aligned}
\end{equation*}
with $\bar\alpha=\alpha(2/R_0)^{N-2}$ as before. Finally, applying \eqref{eq:int-finito} to $T\mathfrak{\tilde{v}}$, we conclude the proof.
\end{proof}

\begin{remark}
We haven't investigated yet neither the optimality of the growth $e^{u^2}$ nor the value of sharp exponent $\alpha^*$ for the validity of inequality \eqref{eq:Trud-Moser-type}. Therefore the problem of finding the maximal growth for integrability of functions in $\mathcal K$ or the sharp Trudinger-Moser inequality in $\mathcal K$ is still open. However, this will not be needed for our purposes.
\end{remark}

It is known that Pohozaev-Trudinger-Moser inequalities imply the validity of the embeddings in the corresponding Orlicz spaces.
For completeness, we report below all the details of the embedding descending from Proposition \ref{cor:Trud-Moser1}. We start with a preliminary lemma.
 
\begin{lemma}\label{lem:expo_order}
Let $\alpha,C>0$. There exists $\tilde\alpha\in (0,\alpha)$ such that
\begin{equation}\label{eq:expo_order}
e^{\tilde\alpha s^2}-1 \le \frac{e^{\alpha s^2}-1}{C}
\quad\mbox{for all }s\in \mathbb{R}.
\end{equation}
\end{lemma}
\begin{proof}
Let $\beta\in \left(0,\min\left\{ \alpha,\alpha/C \right\}\right)$. There exist $\delta,M>0$ such that
\begin{equation}\label{eq:expo_order_aux}
e^{\beta s^2}-1 \le \frac{e^{\alpha s^2}-1}{C}
\quad\mbox{for all } |s|\le\delta \mbox{ and } |s|\geq M.
\end{equation}
Indeed, the existence of $\delta >0$ follows from the Taylor expansion at the origin of the functions on the two sides of the inequality, while the existence of $M>0$ is a consequence of their behaviour at infinity.
Given such $\delta,\,M$, there exists $\tilde\alpha\in \left(0,\beta\right)$ sufficiently small satisfying
\[
e^{\tilde\alpha M^2}-1 \le \frac{e^{\alpha \delta^2}-1}{C},
\]
because $(e^{\alpha \delta^2}-1)/C>0$ and $\lim_{b\to0^+} (e^{b M^2}-1)=0$. Then it follows that
\[
e^{\tilde\alpha s^2}-1\le e^{\tilde\alpha M^2}-1\le \frac{e^{\alpha\delta^2}-1}{C} \le  \frac{e^{\alpha s^2}-1}{C}\quad\mbox{for all }\delta\le|s|\le M.
\]
Since also \eqref{eq:expo_order_aux} holds with $\tilde\alpha$ in place of $\beta$ in view of the monotonicity of the exponential function, this leads to \eqref{eq:expo_order}.
\end{proof}

\begin{proposition}\label{prop:cont-embedding}
The cone $\mathcal K$ is continuously embedded in $\mathcal M(A)$, that is $\mathcal K \subset \mathcal M(A)$ and there exists $C>0$ such that 
\[
\|u\|_{\mathcal L(A)}\le C\|u\|_{H^1_\lambda(A)}\quad\mbox{for all }u\in\mathcal K.
\]
\end{proposition}
\begin{proof} We start observing that, in view of \eqref{eq:int-finito-K}, $\mathcal K\subset \mathcal M(A)$. 

Now, let $u\in\mathcal{K}$ with $\|u\|_{H^1(A)}\le 1$ and let $\alpha^*$ be as in Proposition \ref{cor:Trud-Moser1}.  By Lemma \ref{lem:expo_order} applied with $\alpha=\alpha^*$ and $C=C^*_\mathcal K$, there exists $\tilde \alpha\in (0,\alpha^*)$ such that
\[\int_A (e^{\tilde \alpha u^2} - 1)\, dx\le \frac{1}{C^*_\mathcal K}\int_A (e^{\alpha^* u^2} - 1)\, dx\le 1,\]
which implies that $\|u\|_{\mathcal L(A)}\le 1/\sqrt{\tilde \alpha}$ by the definition of the Luxemburg norm $\|\cdot\|_{\mathcal{L}(A)}$. In conclusion, applying the previous inequality to $u/\|u\|_{H^1(A)}$, we get for all $u\in H^1_0(A)\cap\mathcal K$, 
\[
\|u\|_{\mathcal L(A)}\le \frac{1}{\sqrt{\tilde\alpha}}\|u\|_{H^1(A)}.
\]
Since the norms $\|\cdot\|_{H^1(A)}$ and $\|\cdot\|_{H^1_\lambda(A)}$ are equivalent in $\mathcal K\subset H^1_0(A)$, the continuous embedding stated is proved.
\end{proof}

We recall that the compact embedding of $\mathcal K$ in the Lebesgue spaces $L^q(A)$ with $q>2$ for $R_1=\infty$ has been proved in \cite[Proposition 2.8]{BCNW2024}. Altogether, also in view of the above proposition, $\mathcal K$ verifies assumption (i) of Theorem \ref{thm:abstract}.

\subsection{Pointwise invariance property}
In this subsection, we will prove the validi\-ty of assumption (ii) of Theorem \ref{thm:abstract}. 
We first consider the auxiliary linear problem in the cone $\mathcal K$; arguing as in \cite[Lemma 2.9]{BCNW2024} 
it is possible to prove the following.  
\begin{lemma}\label{le:invarianceK}
Let $h\in \mathcal K\cap L^\infty(A)$. The linear problem 
\begin{equation}\label{Ph}
\begin{cases}
-\Delta v + \lambda v =h \quad&\mbox{in } A\\
v \in H^1_0(A)
\end{cases}
\end{equation}
has a unique solution $v$ and moreover $v\in \mathcal K$.
\end{lemma}

By combining Lemma \ref{le:invarianceK} with a truncation argument, we can conclude the proof of the pointwise invariance of $\mathcal K$. Here assumptions \ref{f6}--\ref{f10} come into play.

\begin{proposition}\label{lem:pointwise-invariance-app}
Let $u\in \mathcal K$ and $f$ satisfy \ref{f1}--\ref{f5}, with $\Omega=A$, and \ref{f6}--\ref{f10}. The problem 
\begin{equation}\label{P-u-1}
\begin{cases}
-\Delta v + \lambda v = f(x,u) \quad&\mbox{in } A\\
v=0& \mbox{on }\partial A
\end{cases}
\end{equation}
has a unique solution $v$ and moreover $v\in \mathcal K$.
\end{proposition}
\begin{proof}
We first prove the result for $u\in \mathcal{K}\cap L^\infty(A)$. Since $u\in \mathcal K$ and 
$f$ satisfies  \ref{f1}, \ref{f6}--\ref{f8}, we have that
$h:=f(x,u)$ has the sign, symmetry, and monotonicity assumptions required in the definition \eqref{eq:K-def} of $\mathcal K$. We want to show that $h\in H^1_0(A)\cap L^\infty(A)$. By \ref{f1}, since $u$ is bounded, $h\in L^\infty(A)$. Moreover, using \eqref{eq:asterisco_mu_nu} with $\nu=0$ and $\tilde\mu=1$, and in view of \eqref{eq:favourite_ineq}, we have for every $\alpha>0$
\[
f(x,u)^2\le \begin{cases}C[1+(e^{2\alpha u^2}-1)]\quad & \mbox{if $\Omega$ is bounded}\\
C[u^2+(e^{2\alpha u^2}-1)] & \mbox{if $\Omega$ is unbounded}
\end{cases}
\]
for a suitable $C>0$. Thus in both cases, since 
$u\in \mathcal K\subset H^1_0(A)$ and using \eqref{eq:int-finito-K}, we have $h\in L^2(A)$.
Furthermore, by \ref{f1} and the fact that $u\in H^1_0(A)$, $h|_{\partial A}=f(\cdot,0)=0$. 
It remains to prove that $|\nabla h|\in L^2(A)$. For this, we use the assumption \ref{f10}, and the fact that $u\in H^1_0(A)\cap L^\infty(A)$ to get 
\[
\begin{split}
\int_A |\nabla h|^2\,dx \le 2\int_A |\nabla_x f(x,u)|^2 \, dx+2\int_A |\partial_s f(x,u)|^2|\nabla u|^2 \, dx\\
\le 2\|d_M\|_{L^2(A)}^2 + 2D_M^2\|\nabla u\|_{L^2(A)}^2 <\infty,
\end{split}
\]
with $M=\|u\|_{L^\infty(A)}$ and $d_M$, $D_M$ as in \ref{f10}. Then, applying Lemma \ref{le:invarianceK} with $h=f(x,u)\in \mathcal K\cap L^\infty(A)$, the thesis is proved in this case. 

Now, for a general $u\in \mathcal{K}$, we will use a truncation argument as in \cite[Proposition 4.2]{CowanMoameni2022}. Let $u\in \mathcal K$ and define $u_n:=\min\{u,n\}$ for every $n\in\mathbb N$. Then, $(u_n)\subset \mathcal K\cap L^\infty(A)$. By the first part of the proof, for every $u_n$ there exists a unique $v_n\in\mathcal K$ solving 
\[
\begin{cases}
-\Delta v + \lambda v = f(x,u_n)\quad&\mbox{in } A\\
v=0&\mbox{on } \partial A.
\end{cases}
\]
Hence, using the equation solved by $v_n$ and H\"older's inequality, we obtain
\begin{equation}\label{eq:est-vn}
\begin{aligned}
\|v_n\|_{H^1_\lambda(A)}^2&= \int_A f(x,u_n)v_n\,dx\le \|f(x,u_n)\|_{L^2(A)}\|v_n\|_{L^2(A)}\\
&\le C\|f(x,u_n)\|_{L^2(A)}\|v_n\|_{H^1_\lambda(A)}.
\end{aligned}
\end{equation}
Now, since $u_n\to u$ a.e. in $A$, $0\le u_n\le u_{n+1}$ for every $n$, $f(x,u_n)\ge 0$ by \ref{f1}, and $f(x,\cdot)$ is nondecreasing by \ref{f8},  
\[
\lim_{n\to\infty}\int_A f(x,u_n)^2\,dx = \int_A f(x,u)^2\,dx
\]
by monotone convergence.
Therefore, $(f(x,u_n))$ is bounded in $L^2(A)$, and so \eqref{eq:est-vn} implies that $(\|v_n\|_{H^1_\lambda(A)})$ is bounded. Thus, there exists $v\in H^1_0(A)$ such that up to a subsequence $v_n\rightharpoonup v$ weakly in $H^1(A)$ and so for every $\varphi\in C^\infty_0(A)$, 
\[
\lim_{n\to\infty}\int_A(\nabla v_n\cdot\nabla \varphi+\lambda v_n\varphi)\,dx = \int_A(\nabla v\cdot\nabla \varphi+\lambda v\varphi)\,dx. 
\]
Moreover, for every $\varphi\in C^\infty_0(A)$ it also holds, by monotone convergence, 
\[
\begin{aligned}
\lim_{n\to\infty}&\left|\int_A \left(f(x,u_n)-f(x,u)\right) \varphi\,dx \right|\\  
& \le \lim_{n\to\infty}\|\varphi\|_{L^\infty(A)} \int_{A}\left(f(x,u) - f(x,u_n)\right)\,dx = 0. 
\end{aligned}
\]
Whence, being for every $n\in\mathbb N$, 
\[
\int_A(\nabla v_n\cdot\nabla \varphi+\lambda v_n\varphi)\,dx = \int_A f(x,u_n)\varphi\,dx\;\mbox{ for every }\varphi\in C^\infty_0(A),
\]
we get 
\[
\int_A(\nabla v\cdot\nabla \varphi+\lambda v\varphi)\,dx = \int_A f(x,u)\varphi\,dx\;\mbox{ for every }\varphi\in C^\infty_0(A)
\]
that is $v\in H^1_0(A)$ solves \eqref{P-u-1}. Finally, since $\mathcal K$ is weakly closed by Lemma \ref{le:Kcone}, $v\in \mathcal K$ and the proof is concluded.
\end{proof}

\subsection{End of the proof of Theorem \ref{thm:applied1}}
The thesis follows from an application of the abstract Theorem \ref{thm:abstract}, with the domain $\Omega=A$ as in \eqref{eq:A-def} and the convex cone $K=\mathcal K$ as in \eqref{eq:K-def}. Notice that $\mathcal K$ is closed by Lemma \ref{le:Kcone}. 
Moreover, by Section \ref{subsec:verification(i)}, the embeddings required in (i) of Theorem \ref{thm:abstract} hold, and by Proposition \ref{lem:pointwise-invariance-app}, the pointwise invariance property (ii) is satisfied. \hfill $\qed$

\begin{remark} \label{rem:prototype}
The prototype exponential nonlinearity defined in \eqref{eq:prototype}, with weight $w$ as in \eqref{eq:hp-w}, satisfies \ref{f1}-\ref{f10}. Indeed, since
\begin{equation} \label{eq:modelexp2}
f(x,s) \sim \frac{w(x)}{m!}|s|^{\beta(m+1)-2}s \quad \text{as } s \to 0,
\end{equation}
$f(x,s) \sim w(x)|s|^{\beta-2}s e^{|s|^\beta}$ as $|s| \to \infty$ and \eqref{eq:hp-w} holds, it is not difficult to check that \ref{f1} and \ref{f2} are satisfied.
Moreover, we have that 
$$F(x,s)=\frac{w(x)} \beta \exp_{m+1} \left( |s|^\beta\right) \quad\mbox{ for all $(x,s)\in A\times\mathbb R$} \: .$$
Concerning \ref{f3}, using that $w>0$ in $A$, we can estimate
\[
s f(x,s) =  w(x) \sum_{i=m+1}^\infty \frac{|s|^{i\beta}i}{i!} \geq w(x) (m+1) \sum_{i=m+1}^\infty \frac{|s|^{i\beta}}{i!} = \beta(m+1) F(x,s)
\]
for all $(x,s)\in A\times \mathbb R$, hence \ref{f3}-(b) is verified with $\sigma= \beta (m+1) > 2$ and any $\overline{M}>0$. In the unbounded case, the first part of \ref{f3}-(u) is verified with the same $\sigma$ as before, and the second part of the assumption is satisfied with $\sigma=\beta(m+1)>2$, $c_1=w/(\beta(m+1)!)\in L^\infty(A)$, and $c_2\equiv0$ because 
\[\begin{aligned}
F(x,s)& = \frac{w(x)}{\beta}\left(\frac{|s|^{\beta(m+1)}}{(m+1)!}+\sum_{i=m+2}^\infty\frac{|s|^{i\beta}}{i!}\right)\ge \frac{w(x)}{\beta(m+1)!}|s|^{\beta(m+1)} 
\end{aligned}
\]
for every $(x,s)\in A\times\mathbb R$.
Furthermore, since $\beta(m+1)>2$,
$$\frac{F(x,s)}{s^2}\sim \frac{w(x)}{\beta(m+1)!}|s|^{\beta(m+1) -2}\to 0 \quad \text{ as } s\to 0 \:,$$
hence
\ref{f5}-(b) holds true if $A$ is bounded. Similarly in view of \eqref{eq:modelexp2}, we have
$$\frac{|f(x,s)|}{|s|^\mu} \sim \frac{w(x)}{m!}|s|^{\beta(m+1)-1 - \mu} \quad \text{as } s \to 0 $$
and \ref{f5}-(u) holds with any $\mu\in (1,\beta(m+1)-1]$, whose existence is ensured by the assumption $\beta(m+1)>2$.
The hypotheses \ref{f6},\ref{f7}, and \ref{f10} are satisfied thanks to \eqref{eq:hp-w} and the fact that $\partial_sf(x,\cdot)\in C([0,\infty))$ for all $x\in A$. Indeed,  for all $x \in A$ and $s \geq 0$, we have
\begin{equation} \label{eq:sf_sbeta<1}
\partial_s f(x,s)=  w(x) \left[ (\beta-1) s^{\beta-2} \exp_m \left( s^\beta \right) + \beta s^{2 \beta-2} \exp_{m-1} \left(s^ \beta \right)\right] \: .
\end{equation}
Finally to check the validity of \ref{f8}, using the series expansion of the exponential function $\exp_{m-1}$, we can estimate
\begin{equation} \label{eq:estimate ofsf_sbeta<1}
\partial_s f(x,s) \geq  w(x)[ \beta(m+1) - 1]s^{\beta-2} \exp_m \left( s^\beta \right) \geq 0  \quad\mbox{ for all $(x,s)\in A\times\mathbb R^+$} \:.
\end{equation}

In the same way, for the inhomogeneous power nonlinearity defined in \eqref{eq:prototype-power}, with weight $w$ satisfying \eqref{eq:hp-w}, 
it is easy to check that all the hypotheses \ref{f1}-\ref{f10} are verified.
In particular, it is enough to recall that $F(x,s)=w(x)(\:|s|^p/p + |s|^{\mathfrak{p}}/{\mathfrak{p}} \: )$ for all $(x,s)\in A\times\mathbb R$, to see that 
\ref{f3} is satisfied with $\sigma=p>2$, $c_1=w/p\in L^\infty(A)$, and $c_2\equiv0$. Moreover, \ref{f5}-(b) is fulfilled since $p>2$ and \ref{f5}-(u) also holds with any $\mu\in (1,p-1]$.
\end{remark}

\section{Symmetry breaking in the radial case}
\label{sec:case-radial}
In this section, we consider problem \eqref{eq:main-appl} with radial nonlinearity $f=f(|x|,s)=\mathfrak{f}(r,s)$, that is under the stronger assumption \ref{f'6} instead of \ref{f6}. In this setting, we prove Theorem \ref{thm_nonradial},  showing that under the additional assumptions \ref{f11}--\ref{f12}, the solution found in Theorem \ref{thm:applied1} is nonradial. 

We start with two technical lemmas about the functional $J$ introduced in \eqref{eq:J-def}, that is
\begin{equation*}
J(u):=\frac12 \|u\|_{H^1_\lambda(A)}^2-\int_A F(x,u) \, dx.
\end{equation*}

\begin{lemma}\label{lem:g}
Under the assumptions of Theorem \ref{thm_nonradial}, for any $\phi\in \mathcal K\setminus\{0\}$, the function $g:t\in[0,\infty)\mapsto J(t\phi)$ satisfies $g(t)<0$ for $t$ sufficiently large. Moreover, $g$ is differentiable, has a unique maximum point $t_\phi\in (0,+ \infty)$, and
\begin{equation}\label{eq:g'-sign}
g'(t)\begin{cases}
>0\quad&\mbox{if }0<t<t_\phi\\
=0\quad&\mbox{if }t=t_\phi\\
<0\quad&\mbox{if }t>t_\phi.
\end{cases}
\end{equation}
\end{lemma}
\begin{proof}
Let $\phi$ be as in the statement. By the definition of $g$ and the mountain pass geometry conditions satisfied by $J$ (see the proof of Theorem \ref{thm:abstract} in Section \ref{sec:abstract_theorem}), $g(0)=0$, $g(t)>0$ for $t>0$ sufficiently small, and $g(t)<0$ for $t$ sufficiently large. 
We notice that $g$ is differentiable by Proposition \ref{prop:J-C1}.  Hence
there exists $t_\phi>0$ such that $\max_{t\ge 0}g(t)=g(t_\phi)$ and $g'(t_\phi)=0$.  
For every $t\ge 0$ we can compute 
\begin{equation}
\label{eq:g'-def}
g'(t)=J'(t\phi)[\phi]=t\int_A\left(|\nabla \phi|^2+\lambda \phi^2-\frac{f(x,t\phi)}{t}\phi\right)\, dx,
\end{equation}
hence 
\[
g'(t)=0 \,\mbox{ with }t>0\quad	\Leftrightarrow\quad \|\phi\|_{H^1_\lambda(A)}^2=\int_A \frac{f(x,t\phi)}{t}\phi\, dx.
\]
Thus, $\|\phi\|_{H^1_\lambda(A)}^2=\int_A \frac{f(x,t_\phi \phi)}{t_\phi}\phi\,dx$.
Now, suppose by contradiction that there is another nonzero critical point of $g$, namely there exists $0<\bar t\neq t_\phi$ such that $g'(\bar t)=0$. Then, 
\begin{equation}\label{eq:uguaglianzaintegrali}
\int_A \frac{f(x,t_\phi \phi)}{t_\phi}\phi\, dx=\|\phi\|_{H^1_\lambda(A)}^2=\int_A \frac{f(x,\bar t \phi)}{\bar t}\phi\,dx.
\end{equation}
On the other hand, by \ref{f1} and \ref{f12} for every $x\in A$ and $s>0$ 
\begin{equation}\label{eq:exf13}
\text{the map }\quad t\in (0,\infty)\mapsto \frac{f(x,ts)}{t}\quad \text{ is strictly increasing}.
\end{equation}
This contradicts \eqref{eq:uguaglianzaintegrali}, thus providing the uniqueness of the maximum point $t_\phi$. Finally, \eqref{eq:g'-sign} follows from \eqref{eq:exf13} and \eqref{eq:g'-def}, which completes the proof.  
\end{proof}

\begin{lemma}\label{lem:C2}
Under the assumptions of Theorem \ref{thm_nonradial},
the functional $J$ is of class $C^2(V;\mathbb R)$.
\end{lemma}
\begin{proof}
The first term, $\|\cdot\|_{H^1_\lambda(A)}^2$, is clearly of class $C^2$. 
By Proposition \ref{prop:J-C1} we already know that $\Phi(\cdot)=\int_A F(x,\cdot) \, dx$ is of class $C^1$,  with $\Phi'(u)[v]=\int_A f(x,u)v \, dx$ for every $u,\,v\in V$.
It remains to prove that $\Phi$ is of class $C^2$. 
We first show that the second G\^ateaux derivative exists, that is
\begin{equation}\label{eq:existsGateaux2}
\begin{aligned}
\lim_{t\to 0}\frac{1}{t}&\left(\Phi'(u+tw)[v]-\Phi'(u)[v]-\Phi''(u)[v,tw]\right)\\
&=\lim_{t\to 0}\frac{1}{t}\int_A[f(x,u+tw)v-f(x,u)v-t\partial_sf(x,u)vw]\, dx =0,
\end{aligned}
\end{equation}
where $\Phi''(u)[v,w]=\int_A \partial_sf(x,u)v w\, dx$ for all $u,\,v,\,w\in V$.
Indeed, we have that
$\frac{1}{t}[f(x,u+tw)-f(x,u)-t\partial_sf(x,u)w]\to 0$ a.e. in $A$,  by the regularity of $f$ required in \ref{f11}.
Moreover, by the Mean Value Theorem and \ref{f11} (where we treat simultaneously the two cases (u) and (b), taking formally $\mu=1$ when $R_1<\infty$), we get for a.e. $x\in A$ and for some $\alpha>0$,
\begin{equation}\label{eq:g-def}
\begin{aligned}
\frac{1}{|t|}&\big|f(x,u+tw)-f(x,u)-t\partial_sf(x,u)w\big||v|\\
&= \big|\partial_sf\big(x,u+t(x)w\big)-\partial_sf(x,u)\big||w||v|\\
&\le 
\big[C|u+t(x)w|^{\mu-1}+(e^{\alpha\left(u+t(x)w\right)^2}-1)+C|u|^{\mu-1}+(e^{\alpha u^2}-1)\big]|w||v|\\
&\le C\left[|u|^{\mu-1}+|w|^{\mu-1}+\frac{e^{4\alpha u^2}-1}2
+\frac{e^{4\alpha w^2}-1}2+(e^{\alpha u^2}-1)\right]|w||v|=:g
\end{aligned}
\end{equation}
for some $t(x)\in (0,1)$.
Now, by H\"older's inequality with exponents 2, 4, and 4, taking into account \eqref{eq:favourite_ineq}, we have  
\[
\int_A (e^{\alpha u^2}-1)|w||v|\, dx\le \left(\int_A (e^{2\alpha u^2}-1)\, dx\right)^{\frac{1}{2}}\|v\|_{L^4(A)}\|w\|_{L^4(A)}<\infty,
\]
where we have used the embedding $\mathcal M(A)\hookrightarrow L^4(A)$ given in Lemma \ref{lem:L-embedded}. Similarly, $\frac{e^{4\alpha u^2}-1}2|w||v|$, $\frac{e^{4\alpha w^2}-1}{2}|w||v|\in L^1(A)$.
If $R_1<\infty$, this is enough to conclude that $g\in L^1(A)$. Otherwise, if $A$ is unbounded, since $\mu>1$, by H\"older's inequality with exponents $\mu+1$, $\mu+1$, and $(\mu+1)/(\mu-1)$, and by Lemma \ref{lem:L-embedded} we get 
\[
\int_A |u|^{\mu-1}|v||w|\, dx\le \|u\|_{L^{\mu+1}(A)}^{\mu-1}\|v\|_{L^{\mu+1}(A)}\|w\|_{L^{\mu+1}(A)}<\infty.
\]
Analogously, $|w|^{\mu-1}|w||v|\in L^1(A)$. 
Hence, also in this case, the function $g$ introduced in \eqref{eq:g-def} belongs to $L^1(A)$. Therefore,
by the Dominated Convergence Theorem, \eqref{eq:existsGateaux2} follows. 

As for the continuity of $\Phi''$, we need to prove that if $(u_n)\subset V$ is such that $u_n\to u$ in $V$, then 
\begin{equation}
\label{eq:continuity-der2}
\lim_{n\to \infty}\sup_{\substack{v\in V\\ \|v\|_V\le 1}}\left|\int_A \partial_s f(x,u_n)v^2 \, dx - \int_A \partial_s f(x,u)v^2 \, dx\right| = 0.
\end{equation}
Applying H\"older's inequality with generic conjugate exponents $q$ and $q'$ to be determined later, we can estimate 
\[
\left|\int_A \left(\partial_s f(x,u_n) -  \partial_s f(x,u)\right) v^2 \, dx\right| \le \|\partial_s f(\cdot,u_n) -  \partial_s f(\cdot,u)\|_{L^q(A)}\|v\|_{L^{2q'}(A)}^2.
\]
Clearly, $\partial_s f(\cdot,u_n) \to \partial_s f(\cdot,u)$ a.e. in $A$, moreover by \ref{f11} (where again we treat simultaneously the two cases (u) and (b), taking formally $\mu=1$ when $R_1<\infty$) we have, for a.e. $x\in A$ and for some $\alpha>0$,
\[
\begin{aligned}
|\partial_s f(x,u_n) -  \partial_s f(x,u)|^q &\le C\left(|\partial_s f(x,u_n)|^q+|\partial_s f(x,u)|^q\right)\\
&\le C\left[|u_n|^{q(\mu-1)}+(e^{q\alpha u_n^2}-1)+|u|^{q(\mu-1)}+(e^{q\alpha u^2}-1)\right]\\
&=:g_n,
\end{aligned}
\]
for a suitable $C>0$. Letting now $g:=2C\left[|u|^{q(\mu-1)}+(e^{q\alpha u^2}-1)\right]$ and arguing as above, we get that $g_n\to g$ a.e. in $A$ and $g_n,\,g\in L^1(A)$, provided
\begin{equation}
\label{eq:choice-of-q}
q\ge 1\mbox{ if $R_1<\infty$ \quad and}\quad q>\max\left\{1,\frac{2}{\mu-1}\right\}\mbox{ if $R_1=\infty$.}
\end{equation}
Since, in particular, $u_n\to u$ in $\mathcal M(A)$, it follows from Lemmas \ref{lem:L-embedded} and \ref{lem:auxiliary} that $\int_A g_n\to \int_A g$. Altogether, by the General Dominated Convergence Theorem, $\partial_s f(\cdot,u_n) \to  \partial_s f(\cdot,u)$ in $L^q(A)$ for $q$ as in \eqref{eq:choice-of-q} and so
\[
\begin{aligned}
\sup_{\substack{v\in V \\ \|v\|_{V\leq1}}}&
\left| \int_\Omega \left( \partial_sf(x,u_n)-\partial_sf(x,u) \right) v^2 \,dx \right|\\
&\le \|\partial_s f(\cdot,u_n) -  \partial_s f(\cdot,u)\|_{L^q(A)} \sup_{\substack{v\in V \\ \|v\|_{\mathcal{L}(\Omega)\leq1}}}\|v\|_{L^{2q'}(A)}^2 
 \\
&\leq C^2 \|\partial_s f(\cdot,u_n) -  \partial_s f(\cdot,u)\|_{L^q(A)} \to 0\quad\mbox{ as $n\to\infty$,}
\end{aligned}
\]
with $C=C(2q')$ as in Lemma \ref{lem:L-embedded}. In conclusion, $\Phi\in C^2(V;\R)$.
\end{proof}

Next we introduce the angular function $\eta(x)=\mathfrak y(\theta)$ given by
\begin{equation}\label{eq:eta-def}
\mathfrak y(\theta):=1-N\sin^2\theta=\cos^2\theta-(N-1)\sin^2\theta, 
\end{equation}
for $\theta\in(0,\pi/2)$,  or equivalently, for $x\neq0$,
\[
\eta(x)=\sum_{i=1}^{N-1} \left(\frac{x_i}{|x|}\right)^2 - (N-1)\left(\frac{x_N}{|x|}\right)^2.
\]
In view of the above expression, it is possible to show that 
\begin{equation}\label{radial-weakly-stable-assumption-4}
    \int_{\mathbb S_r^{N-1}} \eta(x)\,d\sigma = 0\quad\mbox{for every $r>0$,}
\end{equation}
where $\mathbb S_r^{N-1}$
 is the sphere of radius $r$ of $\mathbb R^N$ and $d\sigma= (\cos\theta)^{N-2} r^{N-1} \,d\theta$. 
Indeed, as for every $i,j=1,\ldots,N$, and $r>0$
\[
\int_{\mathbb S_r^{N-1}} x_i^2\,d\sigma = \int_{\mathbb S_r^{N-1}} x_j^2\,d\sigma,
\]
it results that
\begin{equation}\label{eq:int_eta0}
\int_{\mathbb S_r^{N-1}} \eta(x)\,d\sigma 
= \frac{1}{r^2}\int_{\mathbb S_r^{N-1}}\left(\sum_{i=1}^{N-1} x_i^2-(N-1) x_N^2 \right)d\sigma 
 = 0.
\end{equation}
Explicit calculations show that $\mathfrak y$ satisfies the angular monotonicity condition required in $\mathcal K$, precisely
\begin{equation}\label{radial-weakly-stable-assumption-2}
\mathfrak{y}'(\theta) \, <\,  0 \qquad\qquad \text{ for $\theta\in (0,\pi/2)$,}
\end{equation} and that it solves the boundary value problem
\begin{equation}\label{eq:eta}
\begin{cases}
-((\cos\theta)^{N-2}\mathfrak y')'=2N (\cos\theta)^{N-2} \mathfrak y \qquad & \mbox{in } \left(0,\frac{\pi}{2}\right) \\
\mathfrak y'(0)=\mathfrak y'\left(\frac{\pi}{2}\right)=0.
\end{cases}
\end{equation}

The next proposition takes inspiration from both \cite{BCNW2024} and \cite{CowanMoameni2022}.  

\begin{proposition}\label{lem:v_def}
Under the assumptions of Theorem \ref{thm_nonradial},  
let $u_\mathrm{rad}(x)=\mathfrak u_\mathrm{rad}(r)\in\mathcal K$ be a nontrivial radial solution of \eqref{eq:main-appl}, and let $\mathfrak y$ be as in \eqref{eq:eta-def}.
Then, the function $v(x)=\mathfrak{v}(r,\theta):=\mathfrak u_\mathrm{rad}(r)\mathfrak y(\theta)$ belongs to $H^1_0(A)\cap \mathcal M(A)$ and satisfies
\begin{equation}
\label{radial-weakly-stable-assumption-1}
J''(u_\mathrm{rad})[v,v]<0. 
\end{equation}
\end{proposition}
\begin{proof}
As $u_\mathrm{rad} \in \mathcal K \subset H^1_0(A)$ and $\mathfrak{y}, |\nabla \eta|\in L^\infty(A)$,  we infer that $v\in H^1_0(A)$. Moreover, since $u_\mathrm{rad} \in \mathcal K \subset \mathcal M(A)$ and $\eta\in L^\infty(A)$, for all $k>0$ it results
\[
\int_A e^{(v/k)^2}\, dx\le \int_A e^{\left(u\|\eta\|_{L^\infty(A)}/k\right)^2}\, dx<\infty,
\]
that is $v\in \mathcal M(A)$. 

In order to prove \eqref{radial-weakly-stable-assumption-1}, we take advantage of Lemma \ref{lem:C2} and of \eqref{eq:gradient-r-theta} to write
\[
\begin{aligned}
&J''(u_\mathrm{rad})[v,v]=\int_A \left( |\nabla v|^2+\lambda v^2-\partial_s f(|x|,u_\mathrm{rad})v^2\right)\,dx\\ 
&= 2\omega_{N-2}\int_{R_0}^{R_1}\bigg\{\int_0^{\frac{\pi}{2}}\bigg[\mathfrak{u}_\mathrm{rad}'^{\,2}(r)\mathfrak y^2(\theta)+\frac{\mathfrak{u}_\mathrm{rad}^2(r)\mathfrak y'^{\,2}(\theta)}{r^2}\bigg.\bigg.\\
&\bigg.\bigg.\hspace{2cm}+\mathfrak{u}_\mathrm{rad}^2(r)\mathfrak y^2(\theta)\Big(\lambda-\partial_s\mathfrak f(r,\mathfrak{u}_{\mathrm{rad}})\Big)\bigg](\cos\theta)^{N-2} r^{N-1}\,d\theta\bigg\}dr\\
&  = 2\omega_{N-2}\int_{R_0}^{R_1}\left[ \mathfrak{u}_{\mathrm{rad}}\mathfrak f(r,\mathfrak{u}_{\mathrm{rad}})-\mathfrak{u}_{\mathrm{rad}}^{\,2}\partial_s\mathfrak f(r,\mathfrak{u}_{\mathrm{rad}}) \right]r^{N-1}\, dr \cdot \int_0^{\frac{\pi}{2}} \mathfrak y^2(\theta)(\cos\theta)^{N-2}\,d\theta \, \\
& \hspace{2cm}+  2\omega_{N-2} \int_{R_0}^{R_1} \frac{\mathfrak{u}_\mathrm{rad}^2(r)}{r^2} r^{N-1}\, dr \cdot \int_0^{\frac{\pi}{2}} \mathfrak y'^{\,2}(\theta)(\cos\theta)^{N-2}\,d\theta,
\end{aligned}
\]
with $2\omega_{N-2}$ as in \eqref{eq:changeintegral} and where we have used that, since $u_\mathrm{rad}$ solves \eqref{eq:main-appl}, it satisfies
\begin{equation}\label{eq:u-rad-integrated}
\int_{R_0}^{R_1} (\mathfrak{u}_\mathrm{rad}'^{\,2} +\lambda \mathfrak{u}_\mathrm{rad}^{2}) r^{N-1}\, dr = \int_{R_0}^{R_1} \mathfrak{u}_\mathrm{rad}\mathfrak f(r,\mathfrak{u}_\mathrm{rad}) r^{N-1}\, dr.
\end{equation} 
Then we notice that \eqref{eq:eta} implies
\begin{equation}\label{eq:eta-integrated}
\int_0^{\frac{\pi}{2}} \mathfrak y'^{\,2}(\theta)(\cos\theta)^{N-2} \, d\theta = 
2N \int_0^{\frac{\pi}{2}} \mathfrak y^2(\theta)(\cos\theta)^{N-2} \, d\theta.
\end{equation}
Altogether, we find
\begin{equation}\label{eq:I''estimate}
\begin{aligned}
&J''(u_\mathrm{rad})[v,v]\\
 & =
2\omega_{N-2} \int_{R_0}^{R_1} \bigg\{ \bigg[ \mathfrak f(r,\mathfrak{u}_{\mathrm{rad}})-\mathfrak{u}_{\mathrm{rad}}\partial_s\mathfrak f(r,\mathfrak{u}_{\mathrm{rad}}) \bigg]\mathfrak{u}_{\mathrm{rad}}+2N \frac{\mathfrak{u}_\mathrm{rad}^2}{r^2}  \bigg\} r^{N-1} \, dr \\
&\hspace{7cm} \cdot\int_0^{\frac{\pi}{2}} \mathfrak y^2(\theta)(\cos\theta)^{N-2}\, d\theta.
\end{aligned}
\end{equation}
Now, Hardy's inequality gives
\[
\int_{R_0}^{R_1} \frac{\mathfrak{u}_\mathrm{rad}^2}{r^2}r^{N-1}dr\le \left(\frac{2}{N-2}\right)^2\int_{R_0}^{R_1} \mathfrak{u}_\mathrm{rad}'^{\,2}r^{N-1}dr
\]
and, since $R_0>0$ and $u_\mathrm{rad}\neq 0$, we also obtain
\begin{equation}\label{eq:estimate-u}
\int_{R_0}^{R_1} \mathfrak{u}_\mathrm{rad}^{\,2}r^{N-1}dr > R_0^2\int_{R_0}^{R_1} \frac{\mathfrak{u}_\mathrm{rad}^2}{r^2}r^{N-1}dr.
\end{equation}
The last two estimates allow to get 
\begin{equation}
\label{eq:estimate}
\int_{R_0}^{R_1}(\mathfrak{u}_\mathrm{rad}'^{\,2} +\lambda\mathfrak{u}_\mathrm{rad}^{2})r^{N-1}dr \begin{cases} \ge \displaystyle{H\int_{R_0}^{R_1} \frac{\mathfrak{u}_\mathrm{rad}^2}{r^2}r^{N-1}dr}\;&\mbox{if }\lambda=0,\medskip\\
> \displaystyle{H\int_{R_0}^{R_1} \frac{\mathfrak{u}_\mathrm{rad}^2}{r^2}r^{N-1}dr} &\mbox{if }\lambda>0,
\end{cases}
\end{equation}
with $H$ defined in \eqref{eq:def-H}. 
Hence, continuing from \eqref{eq:I''estimate} and taking into account \eqref{eq:u-rad-integrated} we have
\[
\begin{aligned}
J''(u_\mathrm{rad})[v,v]& \le
2\omega_{N-2}\int_0^{\frac{\pi}{2}} \mathfrak y^2(\theta)(\cos\theta)^{N-2} \, d\theta\cdot\\
&\qquad\cdot \int_{R_0}^{R_1}\left[\left(\frac{2N}{H}+1\right)\mathfrak{f}(r,\mathfrak{u}_\mathrm{rad})-\mathfrak{u}_\mathrm{rad}\partial_s\mathfrak{f}(r,\mathfrak{u}_\mathrm{rad})\right]
\mathfrak{u}_\mathrm{rad}r^{N-1}dr,
\end{aligned}
\]
where again the inequality is strict if $\lambda >0$.
By \ref{f12}, the above inequality provides \eqref{radial-weakly-stable-assumption-1}.
\end{proof}

\begin{proof}[Proof of Theorem \ref{thm_nonradial}]
Let $u$ be the solution of \eqref{eq:main-appl} found in Theorem \ref{thm:applied1}. Suppose, by contradiction, that $u$ is radial, so that Proposition \ref{lem:v_def} applies to the function 
\[
v=u\eta=\mathfrak u(r)\mathfrak{y}(\theta)=\mathfrak v(r,\theta),
\] 
with $\eta=\mathfrak{y}(\theta)$ as in \eqref{eq:eta-def}, leading to 
\begin{equation}
\label{eq:J''uvv}
J''(u)[v,v]<0.
\end{equation} 
Hence, since $J'(u)=0$, $v$ is a descent direction of $J$ in a neighborhood of $u$. 

We {\it claim} that there exists a small $\bar\tau>0$ such that, letting $\phi=u+\bar\tau v$, we have 
\begin{equation}\label{eq:claim}
t\phi\in\mathcal K\setminus\{0\}\quad\mbox{for all }t> 0,\quad\max_{t\ge 0}J(t\phi)=J(t_\phi \phi) < J(u),
\end{equation} 
with $t_\phi$ as in Lemma \ref{lem:g}.
This will allow us to construct an admissible path along which the maximum value of the  energy is below the mountain pass level $c=J(u)$, thus contradicting the definition of $c$ in \eqref{eq:c_level_def}.

In order to prove the claim, we first observe that there exists $\varepsilon>0$ so small that $u+\tau v\in\mathcal K$ for every $\tau\in (-\varepsilon,\varepsilon)$. Indeed, since $u\ge 0$ and $u\not\equiv 0$, for $|\tau|$ small we have $0\not\equiv u+\tau v=u(1+\tau\eta)\ge 0$;
moreover $\mathfrak v_\theta(r,\theta)=\mathfrak u(r)\mathfrak{y}_\theta(\theta)\le 0$. Thus, using that $\mathcal K$ is a cone, $t(u+\tau v)\in \mathcal K\setminus\{0\}$ for every $(\tau,t)\in (-\varepsilon,\varepsilon)\times (0,\infty)$. Let now 
\[
\psi: (-\varepsilon,\varepsilon)\times (0,\infty)\to \mathbb R, \quad \psi(\tau,t):=J'(t(u+\tau v))[u+\tau v]
\] 
and note that $\psi$ is of class $C^1$ by Lemma \ref{lem:C2}. Moreover, using that $u$ solves \eqref{eq:main-appl}, we get  $\psi(0,1)=J'(u)[u]=0$ and 
\[
\begin{aligned}
\partial_t \psi(0,1)&=J''(u)[u,u]=\int_A(f(x,u)u-\partial_s f(x,u)u^2)\, dx\\
&\le \int_A(f(x,u)-\delta f(x,u))u \, dx <0, 
\end{aligned}
\]
where in the two inequalites we have used \ref{f12}. Therefore, by the Implicit Function Theorem, there exist $0<\varepsilon_1<\varepsilon$, $\varepsilon_2>0$, and a $C^1$-function $h:(-\varepsilon_1,\varepsilon_1)\to (1-\varepsilon_2,1+\varepsilon_2)$ such that $h(0)=1$ and for all $\tau\in (-\varepsilon_1,\varepsilon_1)$ and for $t\in (1-\varepsilon_2,1+\varepsilon_2)$ we have
\begin{equation}
\label{eq:implicitfunction}
\psi(\tau,t)=0 \quad\Leftrightarrow\quad t=h(\tau).
\end{equation} 
Moreover, $h'(0)=-\frac{\partial_\tau\psi(0,1)}{\partial_t\psi(0,1)}=0$. Indeed, using again that $J'(u)=0$, the contradiction hypotesis $u=\mathfrak u(r)$, assumption \ref{f'6}, the form of $v=u\eta$, and the polar coordinate formula, we get
\[
\begin{aligned}
\partial_\tau\psi(0,1)&= J''(u)[v,u]=\int_A \left(\nabla u\cdot\nabla v+\lambda uv-\partial_sf(x,u)uv\right)\, dx\\
&=\int_A\left(f(x,u)-\partial_s f(x,u)u\right) u\eta \, dx\\
&=\int_{R_0}^{R_1}\left(\int_{\mathbb S_r^{N-1}}\eta(x)d\sigma\right)(\mathfrak f(r,\mathfrak u)-\partial_s \mathfrak f(r,\mathfrak u)\mathfrak u) \mathfrak u   \, dr=0,
\end{aligned}
\]
where in the last equality we have taken \eqref{eq:int_eta0} into account.
Thus, restricting $\varepsilon_1$ if necessary, we have $h(\tau)=1+o(\tau)$ for all $\tau\in (-\varepsilon_1,\varepsilon_1)$ and consequently 
$h(\tau)(u+\tau v)-u=\tau v+o(\tau)$. Using a second-order Taylor expansion centered at $\tau=0$, possibly reducing $\varepsilon_1$ further, we can use  \eqref{eq:J''uvv} to estimate 
\begin{equation}
\begin{aligned}\label{eq:max<}
J(h(\tau)(u+\tau v))-J(u) & =\frac{\tau^2}{2}J''(u)[v,v]+o(\tau^2) \\
& <0\quad\mbox{for all }\tau\in (-\varepsilon_1,\varepsilon_1) \setminus\{0\}.
\end{aligned}
\end{equation}
Now, let $\bar\tau\in (-\varepsilon_1,\varepsilon_1) \setminus\{0\}$ be fixed and $\phi=u+\bar\tau v$ as in \eqref{eq:claim}. By Lemma \ref{lem:g}, we know that there exists a unique $t_\phi\in (0,\infty)$ such that 
\[\max_{t\ge 0}J(t\phi)=J(t_\phi \phi).\]
Moreover, $0=g'(t_\phi)=\psi(\tau,t_\phi)$. Hence, by \eqref{eq:implicitfunction}, $t_\phi=h(\bar\tau)$ and the claim follows by \eqref{eq:max<}. 

We are now in a position to built the admissible path as described at the beginning of the proof. Let $\phi=u+\bar\tau v$ as in \eqref{eq:claim} and let $T>t_\phi$ be so large that $J(T\phi)\le 0$ (see Lemma \ref{lem:g}) and let
$\gamma:[0,1]\to\mathcal K$ be defined as $\gamma(t):=Tt\phi$. We notice that $\gamma\in \Upsilon$, moreover by Lemma \ref{lem:g} and \eqref{eq:claim} we have
\begin{equation*}
\max_{t\in[0,1]}J(\gamma(t))=\max_{t\in[0,T]}J(t\phi)=J(t_\phi \phi)=\max_{t\ge 0}J(t\phi)<J(u).
\end{equation*}
This contradicts the definition of $c$ in \eqref{eq:c_level_def}, thus concluding the proof. 
\end{proof}

Finally,  we consider explicitly the prototype cases and prove Corollary \ref{cor}.
Notice that Theorem \ref{thm:main-exp} follows from Corollary \ref{cor}, since $w\equiv 1$ satisfies \eqref{eq:hp-w} and, when $w\equiv 1$, $\beta \in (1,2)$ and $m=1$, the exponential-type nonlinearity \eqref{eq:prototype}  coincides with the nonlinear term in problem \eqref{eq:main-exp}.

\begin{proof}[Proof of Corollary \ref{cor}]
We verified in Remark \ref{rem:prototype} that both \eqref{eq:prototype} and \eqref{eq:prototype-power} satisfy \ref{f1}-\ref{f10}.  
Let us check the validity of assumptions \ref{f11} and  \ref{f12}.

\noindent (i) Consider the exponential-type nonlinearity \eqref{eq:prototype}. In view of \eqref{eq:sf_sbeta<1}, we have that $\partial_sf(x,s) \sim w(x) \beta s^{2 \beta -2} e^{s^\beta}$ as $s \to \infty$. As $\beta<2$,  we infer that for all $\alpha>0$ there exists $M=M(\alpha)>0$ such that $\partial_s f(x,s)\le e^{\alpha s^2}-1$ for $s>M$. 
Now, let $\mu$ be the constant appearing in assumption \ref{f5}-(u).  By Remark \ref{rem:prototype}, we have $\mu\in (1,\beta(m+1)-1]$. Since
$$\partial_sf(x,s) \sim w(x) \frac{\beta(m+1)-1}{m!} s^{\beta(m+1) - 2} \quad \text{ as } s \to 0 \:,$$
also using \eqref{eq:hp-w}, we deduce that there exists $\varepsilon>0$ such that
$$\partial_sf(x,s)\le \frac{\beta(m+1)-1}{m!} \|w\|_{L^\infty(\Omega)}s^{\mu-1} \quad \text{ for }0<s<\varepsilon \: .$$
Finally, $\partial_s f(x,s)\le D_M \le D_M s^{\mu-1}/\varepsilon^{\mu-1}$ for $\varepsilon<s<M$, by \ref{f10}. Altogether, both \ref{f11}-(b) and \ref{f11}-(u) hold true.

Concerning \ref{f12}, let
\begin{equation}\label{eq:delta}
\delta \in \left(\frac{2N}{H}+1,\beta(m+1)-1\right]
\;\mbox{ if }\lambda=0,\quad
\delta = \frac{2N}{H}+1 \; 
\mbox{ if }\lambda>0.
\end{equation}
Notice that, in the case $\lambda=0$, the above interval is non empty thanks to the hypothesis $\beta(m+1)-1>2N/H+1$.
Then, as in \eqref{eq:estimate ofsf_sbeta<1}, we can estimate
$$\partial_s f(x,s) s - \delta f(x,s) \geq w(x) \left[\beta(m+1) - 1 - \delta \right] s^{\beta-1} \exp_m \left( s^\beta \right) \geq 0 \quad \text{ for all } s>0,$$
where in the last inequality we used \eqref{eq:delta}. Thus, also \ref{f12} is proved.

\noindent (ii) As for the inhomogenous power-type nonlinearity \eqref{eq:prototype-power}, one can proceed similarly to the proof of (i).  Notice that in this case \ref{f12} is ensured provided that $p$ satisfies the condition in (ii) of Corollary \ref{cor}.
\end{proof}

\appendix
\section{}
\begin{proposition}\label{prop:zero-measure}
Let $\Omega\subseteq\mathbb R^N$ be a measurable set, $q\in[1,\infty)$, and $u\in L^q(\Omega)$. Then the level set $\{u=M\}:=\{x\in\Omega\,:\, u(x) = M\}$ has zero measure for all $M\in\mathbb R$, except for at most a countable set of values of $M$.  
\end{proposition}
\begin{proof}
Let $S:=\{M\in\mathbb R\setminus \{0\}\,:\,|\{u=M\}|>0\}$ and 
\[
S\left(\frac{1}{n}\right):=\left\{M\in\mathbb R\,:\,|M|>\frac{1}{n},\;|\{u=M\}|>\frac{1}{n}\right\}\quad	\mbox{for every }n\in\mathbb N.
\]
We observe that 
\begin{equation}
\label{eq:Sunion}
S=\bigcup_{n\in\mathbb N}S\left(\frac{1}{n}\right).
\end{equation}
We {\it claim} that, for every $n\in\mathbb N$, the set $S\left(\frac{1}{n}\right)$ is finite. Indeed, suppose by contradiction that there exists $n_0\in\mathbb N$ such that $S\left(\frac{1}{n_0}\right)$ contains an infinite countable set $\{M_j\}_{j\in\mathbb N}\subset S\left(\frac{1}{n_0}\right)$. Then, for $u\in L^q(\Omega)$ as in the statement, we can estimate 
\begin{equation*}
\begin{aligned}
\int_\Omega |u|^q dx \ge \sum_{j\in\mathbb N}\int_{\{u=M_j\}}|u|^q dx = \sum_{j\in\mathbb N}|\{u=M_j\}||M_j|^q > \sum_{j\in\mathbb N}\frac{1}{n_0^{q+1}}=\infty.
\end{aligned}
\end{equation*}
This contradicts the fact that $u\in L^q(\Omega)$, thus proving the claim. 

Altogether, since by \eqref{eq:Sunion} $S$ is a countable union of finite sets, it is countable. 
\end{proof}

\begin{proposition}\label{prop:chi}
Let $\Omega\subseteq\mathbb R^N$ be a measurable set, $q\in[1,\infty)$, and $(u_n) \subset L^q(\Omega)$ such that $u_n\to u$ a.e. in $\Omega$ for some $u \in L^q(\Omega)$. Let $M\in \mathbb R$ be such that $|\{u=M\}| = 0$. Then,  
\[
\chi_{\{u_n<M\}}\to \chi_{\{u<M\}} \quad\mbox{a.e. in $\Omega$,} 
\]
where $\chi$ is the characteristic function in \eqref{eq:chi-def}.
\end{proposition}
\begin{proof}
Let $x\in\Omega$ be such that $u(x)<M$. Then $\chi_{\{u<M\}}(x)=1$ and, for $n$ large, $u_n(x)<M$, so that $\chi_{\{u_n<M\}}(x)=1$.
Similarly, if $x\in\Omega$ is such that $u(x)>M$, then for $n$ sufficiently large $\chi_{\{u_n<M\}}(x)=0=\chi_{\{u<M\}}(x)$. Altogether, since by assumption $|\{u=M\}| = 0$, we can conclude that $\chi_{\{u_n<M\}}\to \chi_{\{u<M\}}$ a.e. in $\Omega$.  
\end{proof}

\section*{Acknowledgments}
This research was partially supported by GNAMPA - INdAM. 
A. Boscaggin and F. Colasuonno were partially supported by the MUR-PRIN project 20227HX33Z funded by the European Union - Next Generation EU. 
B. Noris was supported by the MUR-PRIN project no. 2022R537CS ``NO$^3$'' granted by the European Union -- Next Generation EU.
\noindent

\bibliographystyle{abbrv}
\bibliography{biblio}

\end{document}